\documentclass{amsart}

\makeatletter
\@namedef{subjclassname@1991}{2020 Mathematics Subject Classification}
\makeatother

\usepackage{soul}

\usepackage{color}
\definecolor{mygray}{gray}{.75}
\usepackage{xcolor}
\usepackage[textsize=tiny]{todonotes}
\usepackage{paralist}
\usepackage{hyperref}
\usepackage{amssymb}
\usepackage{tikz,pgfplots}
\usepackage{csquotes}
\usepackage{bbm}
\usepackage{MnSymbol} 

\usepackage[font=small,labelfont=bf]{caption}
\usepackage{subcaption}

\usepackage{algpseudocode}
\usepackage{algorithm}

\usepackage{arydshln}

\theoremstyle{plain}
\newtheorem{thm}{Theorem}[section]
\newtheorem{rem}[thm]{Remark}

\newtheorem{lemma}[thm]{Lemma}
\newtheorem{example}[thm]{Example}
\newtheorem{definition}[thm]{Definition}

\newtheorem{?}[thm]{Problem}

\newcommand{\npar}{n_{\text{par}}} 
\newcommand{\mpar}{m_{\text{par}}} 

\newcommand{\cF}{\mathcal{F}}

\newcommand{\cI}{\mathcal{I}}

\newcommand{\cN}{\mathcal{N}}
\newcommand{\cO}{\mathcal{O}}
\newcommand{\cP}{\mathcal{P}}

\newcommand{\cV}{\mathcal{V}}

\newcommand{\ba}{\mathbf{a}}
\newcommand{\bb}{\mathbf{b}}
\newcommand{\bc}{\mathbf{c}}
\newcommand{\bd}{\mathbf{d}}
\newcommand{\bbf}{\mathbf{f}}
\newcommand{\bg}{\mathbf{g}}
\newcommand{\bi}{\mathbf{i}}
\newcommand{\bj}{\mathbf{j}}
\newcommand{\bp}{\mathbf{p}}
\newcommand{\bq}{\mathbf{q}}

\newcommand{\by}{\mathbf{y}}

\newcommand{\bA}{\mathbf{A}}
\newcommand{\bB}{\mathbf{B}}
\newcommand{\bC}{\mathbf{C}}

\newcommand{\bP}{\mathbf{P}}

\newcommand{\N}{\mathbb{N}}
\newcommand{\R}{\mathbb{R}}
\newcommand{\mS}{\mathbb{S}}

\newcommand\balpha{{\boldsymbol{\alpha}}}

\makeatletter
\newcommand{\opnorm}{\@ifstar\@opnorms\@opnorm}
\newcommand{\@opnorms}[1]{%
  \left|\mkern-1.5mu\left|\mkern-1.5mu\left|
   #1
  \right|\mkern-1.5mu\right|\mkern-1.5mu\right|
}
\newcommand{\@opnorm}[2][]{%
  \mathopen{#1|\mkern-1.5mu#1|\mkern-1.5mu#1|}
  #2
  \mathclose{#1|\mkern-1.5mu#1|\mkern-1.5mu#1|}
}
\makeatother
\makeatletter
\DeclareFontFamily{OMX}{MnSymbolE}{}
\DeclareSymbolFont{MnLargeSymbols}{OMX}{MnSymbolE}{m}{n}
\SetSymbolFont{MnLargeSymbols}{bold}{OMX}{MnSymbolE}{b}{n}
\DeclareFontShape{OMX}{MnSymbolE}{m}{n}{
    <-6>  MnSymbolE5
   <6-7>  MnSymbolE6
   <7-8>  MnSymbolE7
   <8-9>  MnSymbolE8
   <9-10> MnSymbolE9
  <10-12> MnSymbolE10
  <12->   MnSymbolE12
}{}
\DeclareFontShape{OMX}{MnSymbolE}{b}{n}{
    <-6>  MnSymbolE-Bold5
   <6-7>  MnSymbolE-Bold6
   <7-8>  MnSymbolE-Bold7
   <8-9>  MnSymbolE-Bold8
   <9-10> MnSymbolE-Bold9
  <10-12> MnSymbolE-Bold10
  <12->   MnSymbolE-Bold12
}{}

\let\llangle\@undefined
\let\rrangle\@undefined
\DeclareMathDelimiter{\llangle}{\mathopen}%
                     {MnLargeSymbols}{'164}{MnLargeSymbols}{'164}
\DeclareMathDelimiter{\rrangle}{\mathclose}%
                     {MnLargeSymbols}{'171}{MnLargeSymbols}{'171}
\makeatother
\pgfplotsset{compat=1.11}

\DeclareMathOperator*{\Span}{span}

\numberwithin{equation}{section}

\pgfplotsset{select coords between index/.style 2 args={
    x filter/.code={
        \ifnum\coordindex<#1\fi
        \ifnum\coordindex>#2\fi
    }
}}

\title[Linear/Ridge expansions]{Linear/Ridge expansions: Enhancing linear approximations by ridge functions}%
\author{Constantin Greif}
\address{Ulm University, 
	Institute for Numerical Mathematics, 
	Helmholtzstr.\ 18, 89081 Ulm (Germany), 
	{\{constantin.greif,karsten.urban\}@uni-ulm.de}}		
\author{Philipp Junk}
\address{Justus-Liebig-University Giessen, 
	Lehrstuhl Numerische Mathematik, 
	Arndtstr.\ 2, 
	35392 Giessen  (Germany), 
	{Philipp.Junk@math.uni-giessen.de}}
\author{Karsten Urban}
\date{\today}

\begin{document}

\begin{abstract}
	We consider approximations formed by the sum of a linear combination of given functions enhanced by ridge functions -- a Linear/Ridge expansion. For an explicitly or implicitly given function, we reformulate finding a best Linear/Ridge expansion in terms of an optimization problem. We introduce a particle grid algorithm for its solution. Several numerical results underline the flexibility, robustness and efficiency of the algorithm.
	
	One particular source of motivation is model reduction of parameterized transport or wave equations. We show that the particle grid algorithm is able to produce a Linear/Ridge expansion  as an efficient nonlinear model reduction.
\end{abstract}

\maketitle

\section{Introduction}
\label{Sec:1}
Many (numerical) approximations rely on linear approximation schemes. For some $f\in H$, $H$ a normed space, one seeks for a possibly good finite-dimensional subspace $H_N\subset H$ and  an approximation $f_N\in H_N$ of $f$. Examples include finite element, finite volume, spectral or discontinuous Galerkin methods. 
In many cases, such linear schemes work very well, in particular if the error of an approximation scheme can be shown to be bounded in terms of the error of the best approximation in $H_N$ (e.g.\ by the famous Ce\'a lemma, \cite{Cea,MR1971217}). If then the error of the best approximation decays fast (e.g.\ by using a Cl\'ement-type operator, \cite{clement1975approximation}), one obtains an efficient (numerical) approximation. 

Of course, one would want to determine a \enquote{best-possible} approximation. For linear approximation schemes, the worst best possible error is known as the \emph{Kolmorogov $N$-width} $d_N(\cF)$ which is defined for a class $\cF\subset H$ of elements as (\cite{Kolmogorov,MR774404})
\begin{align*}
	d_N(\cF) := \inf_{H_N\subset H; \dim(H_N)=N}\, \sup_{f\in\cF}\, \inf_{f_N\in H_N} \| f-f_N\|_H.
\end{align*}
The class $\cF$ could e.g.\ be a smoothness class (a Sobolev or Besov space) or a set of solutions for certain problems with different data (e.g.\ a parametric partial differential equation (PPDE) for different parameter values, see below).

The question if $d_N(\cF)$ decays \enquote{fast} as $N\to\infty$ (e.g.\ $d_N(\cF)=\cO(e^{-N})$) or \enquote{slow} (e.g.\ $d_N(\cF)\simeq N^{-s}$, $0<s<1$) typically depends on some measure of smoothness of the elements of $\cF$, for example the Besov regularity. If $\cF$ is given as a set of solutions of a problem (such as a PPDE), then the decay of $d_N(\cF)$ is a property of the \emph{problem} at hand (not its discretization). This means that linear approximation schemes are not appropriate for problems with a poor decay of the $N$-width. Hence, other schemes are needed, which are necessarily nonlinear.

There is a huge variety of nonlinear approximation schemes, an extensive list goes far beyond the scope of this introduction. For reasons to be described below, we are interested in ridge functions for building nonlinear approximation schemes. More specific, we aim at constructing an approximation scheme consisting of a \emph{linear part} and some \emph{nonlinear enhancement} in terms of a ridge function update. Doing so, such an expansion consisting of a linear and a nonlinear part, we will be able to treat problems with fast and slow decay of $d_N(\cF)$ by the same algorithm. We need to collect some notation in order to motivate our choice. 

\subsection{Linear/Ridge approximation}
We consider a given function $u:\Omega\to\R$, where $\Omega\subset\R^d$ is an open bounded domain and $u\in L_2(\Omega)$ is the minimal requirement, sometimes we need to assume more, e.g.\ $u\in C^{0,1}(\bar\Omega)$, at least piecewise.\footnote{$C^{0,1}(\bar\Omega):=\{ v\in C(\bar\Omega):\, \exists L>0:\,  \| v(x)-v(y)\|\le L\, \| x-y\|\, \text{for all } x,y\in\Omega\}$.} This function can be given either explicitly or implicitly as the (unknown) solution to a given problem. In order to formulate the approximation problem under consideration, let
\begin{align*}
	X_N := \operatorname{span}(\Phi_N) \subset L_2(\Omega),
	\qquad
	\Phi_N:=\{ \varphi_1,...,\varphi_N\}
\end{align*}
be a given linear space of dimension $N\in\N$ with $\varphi_i$, $i=1,...,N$, being given functions (which do not need to be linearly independent). Next, we recall the notion of a ridge function and refer e.g.\ to \cite{BUHMANN1999103,KOLLECK201535,pinkus2015} for overviews and details.

\begin{definition}
	Let $a\in\R^d$, $b\in\R$ and $v: \R \to\R$. Then, $w:\Omega\to\R$, $w(x):=v(a^\top x+b)$ is called \emph{ridge function} with \emph{profile} $v $, \emph{direction} $a$ and \emph{offset} $b$.
\end{definition}

\begin{rem}
	In practical application, we shall assume $v\in L_2(\R)\cap C^{0,1}_{\operatorname{pw}}(\R)$  for the profiles to be discussed below. We will denote its argument by $\xi\in\R$, i.e., $v(\xi)$.
\end{rem}

In addition to $\Phi_N$, we assume that we are given a finite number $M\in\N$ of profiles
\begin{align*}
	\cV_M:=\{ v_1,...,v_M\} \subset L_2(\Omega)
\end{align*}
and consider the approximation problem
\begin{align}\label{eq:approx}
	u(x) \approx \sum_{i=1}^N \alpha_i\, \varphi_i(x) + \sum_{j=1}^M c_j\, v_j(a_j^\top x + b_j) =: u_\delta(x),
	\qquad x\in\Omega.
\end{align}
A function $u_\delta$ of type \eqref{eq:approx} will be called \emph{Linear/Ridge} expansion. Clearly, such a Linear/Ridge approximation depends on the choices of $\Phi_N$ and $\cV_M$ as well as on the coefficients $\alpha_i, c_j\in\R$, the directions $a_j\in\R^d$ and the offsets $b_j\in\R$. However, in this paper we view $\Phi_N$ and $\cV_M$ to be given (e.g.\ by some preprocessing training), so that we only consider the dependency of $u_\delta$ on the directions, offsets and coefficients. In order to shorten notation, we collect this dependency in one index $\delta$.  The main topic of this paper is thus to investigate appropriate choices of coefficients, directions and offsets in \eqref{eq:approx} in order to determine a \enquote{good} approximation of a given function $u$. 

\subsection{Motivation}
The main source of motivation for this paper comes from model order reduction of parameterized partial differential equations (PPDE) by means of the reduced basis method (RBM, \cite{Haasdonk:RB, RozzaRB,QuarteroniRB}). In order to briefly review it, let $L(\mu)$ be a parameterized partial differential operator and $f(\mu)$ be some given right-hand side, where $\mu\in\cP\subset\R^Q$ is some parameter. One seeks for the exact solution $u(\mu)$ of $L(\mu) u = f(\mu)$ -- in a suitable weak sense. Typically, a suitable (i.e., sufficiently fine, we say \enquote{detailed}) discretization is available. However, the computation of a numerical detailed approximation $u^\cN(\mu)$ (with $\cN\gg 1$ degrees of freedom) is too costly at least for approximating $u(\mu)$ for several different values of $\mu$ (\enquote{multi-query} context) and/or extremely fast (e.g.\ in realtime and/or in an embedded system).

The RBM aims at constructing a linear subspace $H_N\subset H$ (if the PPDE is posed on $H$, e.g.\ the Sobolev space $H^1_0(\Omega)$) of dimension $\N\ni N\ll\cN$, which is typically constructed in an offline training phase. Then, in the online realtime or multiquery environment, a reduced approximation $u_N(\mu)$ is computed very rapidly -- as the (Petrov-Galerkin) projection of the detailed solution $u^\cN(\mu)$ onto $H_N$. 
In this context, the above mentioned class reads
\begin{align*}
	\cF = \{ u(\mu):\, \mu\in\cP\}
\end{align*}
and this explains the particular relevance of the decay of $d_N(\cF)$ for model order reduction, e.g.\ \cite{MR3584545,ConvergenceRB}. In particular, it is known that certain elliptic problems admit an exponential decay of $d_N(\cF)$, \cite{MR2877366}, whereas parameterized transport and wave equations show poor decay, \cite{GREIF2019216,OR16}, which motivates the construction of nonlinear model reduction (approximation) schemes in particular for such types of problems, \cite{
Black_2020,
nonlinearmordevorecohen,
Ehrlacher2019NonlinearMR,
transportedsnapshots, 
OHLBERGER2013901,
welper2019transformed}, 
just to mention a few.

We are suggesting a nonlinear update in terms of ridge functions due to several reasons. First of all, ridge functions have a simple structure and turn out the be particularly suitable for parameterized transport and wave-type problems as we shall see below. Second, there is a rich literature for approximation theory with ridge functions, see, e.g.\ \cite{BUHMANN1999103,RFConstantine,KOLLECK201535,pinkus2015}. Finally, ridge functions are closely related to neural networks, which just recently have been investigated for nonlinear model reduction, see e.g.\ \cite{DALSANTO2020109550,Frescaetal,RBNNGitta}.

\subsection{Outline}
The remainder of this paper is organized as follows. In Section \ref{Sec:2}, we collect all required preliminaries and introduce the arising optimization problem. Section \ref{Sec:3} is devoted to the association of the desired quantities (directions and offsets) with particles, which is the basis for the particle grid algorithm which we introduce in Section \ref{Sec:4}. Some results of our various numerical experiments are presented in Section \ref{Sec:5}. We finish with an outlook in Section \ref{Sec:6}.

\section{Preliminaries}
\label{Sec:2}

In order to quantify what has to be understood by `good' approximation, define the set
 \begin{align}\label{eq:approxset}
	U_{N,M} :=
	\bigg\{ 
	u_\delta(x)
	= \sum_{i=1}^N \alpha_i\, \varphi_i(x) + \sum_{j=1}^M c_j\, v_j(a_j^\top x + b_j),\, \ 
	\alpha_i, b_j, c_j\in\R,\, a_j\in\R^d
	\bigg\},
\end{align}
which is a nonlinear subset of $L_2(\Omega)$. We consider the error in the $L_2$-norm and are interested in the/a best approximation in this norm abbreviated by $\|\cdot\|_0\equiv\|\cdot\|_{L_2(\Omega)}$. In order to make (at least an approximation to) a best approximation accessible, we collect all variables by setting $\balpha:=(\alpha_i)_{i=1,...,N} \in\R^N$, $\ba:=(a_j)_{j=1,...,M}\in\R^{dM}$, $\bb:=(b_j)_{j=1,...,M}$, $\bc:=(c_j)_{j=1,...,M}\in\R^M$ and 
\begin{align}\label{eq:coeffs}
	\delta := (\balpha, \ba, \bb, \bc) \in \R^{N+(d+2)M},
\end{align}
write $u_\delta(x)$ as in \eqref{eq:approxset} and consider the \emph{cost function}
\begin{align}\label{eq:costfunction}
	J_u:\R^{N+(d+2)M}\to \R_{\geq 0},
	\qquad
	J_u(\delta) := \| u - u_\delta\|_0^2.
\end{align}
Then, $u_{\delta^*} = \arg \inf_{u_\delta\in U_{N,M}} \| u - u_\delta\|_0$, where $\delta^* =\arg \inf_{\delta\in\R^{N+(d+2)M}} J_u(\delta)$.

Before we continue let us detail an example which also indicates our particular interest in such approximations to be considered here.

\begin{example}[Parametric linear transport problem]\label{Ex:1}
	Consider the homogeneous linear transport equation on the real line with velocity $\mu > 0$, i.e., $\partial_{t} u(t,x)  + \mu\, \partial_{x} u(t,x)  = 0$ on $(0,T) \times \R$, $T>0$ with initial condition $u(0,x) = u_0(x)$, $x\in\R$. In a model reduction framework, one would interpret $\mu$ as a parameter and would want to approximate the solution $u(t,x;\mu)= u_0(x- \mu t)$ either for many velocities $\mu$ and/or in realtime. This is typically done in terms of a linear combination of \enquote{snapshots} $\varphi_i:=u(\cdot,\cdot\,;\mu^{(i)})$, $i=1,...,N$, where the snapshot parameters $\mu^{(i)}$ are chosen in some offline training phase. However, it is known from \cite{OR16} that the Kolmogorov $N$-width is at most $\cO(N^{-1/2})$, which makes such a linear model reduction inefficient.

However, choosing the profile $v_1:=u_0$ and setting $a_1:=(-\mu,1)^\top$, $b_1:=0$, $c_1:=1$ yields $c_1 v_1(a_1^\top (t,x)+b_1) = u_0(x- \mu t)$, i.e., the exact solution. Hence, choosing $M=1$, $\cV_1=\{ v_1\}$, we get that $\inf_\delta J_u(\delta)=0$ independent of the choice of $\Phi_N$.
\end{example}

We stress the fact that we cannot assume in general that the functions $v_1,...,v_M$ are linearly independent, not even pairwise. There are even cases, where $v_i=v_j$ for some distinct indices $i\not=j$ as the following example shows.
\begin{example}[Parametric wave equation]\label{Ex:2}
	Consider the linear wave equation $\partial^2_{tt} u - \mu^2 \, \partial^2_{xx} u=0$ for $t>0$ and $x\in\R$ with initial conditions $u(0)=u_0$ and $\dot{u}(0)=0$. The solution is given by the famous d'Alembert formula as $u(t,x;\mu)=\frac12 ( u_0(x-\mu t) + u_0(x+\mu t))$. Hence, choosing $v_1=v_2=u_0$, $c_1=c_2=\frac12$, $b_1=b_2=0$ as well as $a_1=(-\mu,1)^\top$ and $a_2=(\mu,1)^\top$ yields the exact solution.
	This is an example of two identical profiles causing an exact representation using different directions.
	Besides, also for $\dot{u}(0) \neq 0$, the wave equation is a sum of two, but then different, ridge functions.
	
	Also this problem is particularly interesting since it is known that projection-based (i.e., linear) model reduction techniques do not work in the sense that the decay of the Kolmogorov $N$-width is at most $\cO(N^{-1/4})$, \cite{GREIF2019216}. However, d'Alembert's solution formula is a ridge function.
\end{example}

These examples should motivate the consideration of the optimization problem
\begin{align}\label{eq:OptProb}
	J_u(\delta) \to \min!
	\quad \delta\in\R^{N+(d+2)M}
\end{align}
for a given function $u\in L_2(\Omega)$.

\begin{lemma} \label{lem-AN}
	Let $u\in L_2(\Omega)$. Then, there exists a minimizer $\delta^*$ of \eqref{eq:OptProb} for $J_u$ defined in \eqref{eq:costfunction}.
\end{lemma}
\begin{proof}
	Since $L_2(\Omega)$ is a Hilbert space and $U_{N,M}\subset L_2(\Omega)$ is closed, the claim follows by standard arguments.
\end{proof}

Note, that $U_{N,M}$ is not necessarily convex and therefore the minimizer in Lemma \ref{lem-AN} is not necessarily unique.
We are now going to determine an approximation $u_\delta\in U_{N,M}$ to the given function $u$ step by step. 

\begin{rem}
In practice we might just have access to the function $u$ indirectly by a quantity like the residuum of a PDE. In this case, one would not use the $L_2$-error as cost function $J_u $ but rather the residuum (if computable) or an appropriate error estimator.
\end{rem}
\subsection{Given directions and offsets}
As a first step, let us assume that the directions $\ba$ and offsets $\bb$ according to the profiles $\cV_M$ would be given. Then, the optimal coefficients $\balpha$ and $\bc$ for $\ba$, $\bb$ are easily determined as we shall see next.

\begin{lemma}\label{La:optcoeff}
	Let $u\in L_2(\Omega)$ and $\Phi_N$, $\cV_M$, $\ba =(a_j)_{j=1,...,M}\in\R^{dM}$ as well as $\bb =(b_j)_{j=1,...,M}$ be given. 
	Define the matrices $\bA_{i,i'} := (\varphi_i, \varphi_{i'})_0$, $i,i'=1,...,N$ 
		and $\bC_{j,j'} = (\bC(\ba,\bb))_{j,j'} := ( v_j(a_j^\top\cdot+b_j), v_{j'}(a_{j'}^\top\cdot+b_{j'}))_0$ 
		as well as $\bB_{i,j} = (\bB(\ba,\bb))_{i,j} := ( \varphi_i, v_{j}(a_{j}^\top\cdot+b_j))_0$, $i,i'=1,...,N$, $j,j'=1,...,M$. 
	Then, the optimal coefficients 
	$\balpha^*=\balpha^*(u,\ba,\bb)=(\alpha_1^*,...,\alpha^*_N)^\top$ 
	and $\bc^*=\bc^*(u,\ba,\bb)=(c_1^*,...,c_M^*)^\top$ 
	are given as the solution of the linear system of equations
	\begin{align}\label{eq:LinSys}
		\begin{pmatrix}
			\bA 			& \bB \\
			\bB^\top 	& \bC
		\end{pmatrix}
		\begin{pmatrix}
			\balpha \\
			\bc
		\end{pmatrix}
		= 
		\begin{pmatrix}
			\bbf(u) \\
			\bg(u)
		\end{pmatrix},
	\end{align}	
	where the right-hand side is given by 
	$\bbf(u)_i:= (u,\varphi_i)_0$, $i=1,...,N$,  
	and $\bg(u, \ba, \bb)_j :=( u, v_{j}(a_{j}^\top\cdot+b_j))_0$, $j=1,...,M$.
\end{lemma}
\begin{proof}
	Let $\ba$, $\bb$ be fixed and denote $\tilde\delta:= (\balpha, \bc)$ as well as $\delta=(\balpha, \ba, \bb, \bc)$ as in \eqref{eq:coeffs}.  We consider the reduced cost function $\tilde{J}_u(\tilde\delta):= J_u(\delta)$ as a function of $\tilde\delta$ only. Then, $\tilde J_u(\tilde\delta) = \| u-u_{\tilde\delta}\|_0^2= \| u\|_0^2 - 2(u,u_{\tilde\delta})_0 + \| u_{\tilde\delta}\|_0^2$ and
	\begin{align*}
		(u,u_{\tilde\delta})_0
		&= \sum_{i=1}^N \alpha_i\, (u,\varphi_i)_0
			+ \sum_{j=1}^M c_j\, (u, v_j(a_j^\top\cdot+b_j))_0
			= \bbf(u)^\top \balpha + \bg(u)^\top \bc,\\
		\| u_{\tilde\delta}\|_0^2
		&= \sum_{i,i'=1}^N \alpha_i\, \alpha_{i'} (\varphi_i,\varphi_{i'})_0
			+ \sum_{j,j'=1}^M c_j\, c_{j'}\, (v_j(a_j^\top\cdot+b_j), v_{j'}(a_{j'}^\top\cdot+b_j))_0\\
		&\qquad + 2 \sum_{i=1}^N \sum_{j=1}^M\alpha_i\,c_j (\varphi_i,v_j(a_j^\top\cdot+b_j))_0 
		= \balpha^\top \bA \balpha + \bc^\top \bC \bc
			+ 2 \balpha^\top \bB \bc,
	\end{align*}
	so that $\tilde{J}_u({\tilde\delta}) = \| u\|_0^2 - 2\, \bbf(u)^\top \balpha - 2\, \bg(u)^\top \bc + \balpha^\top \bA \balpha + \bc^\top \bC \bc + 2\, \balpha^\top \bB \bc$. 
	Hence, since $\bA$ and $\bC$ are symmetric,
	\begin{align*}
		\nabla \tilde{J}_u(\tilde\delta)
		&= \begin{pmatrix} \partial_{\balpha_N} J_u(\delta) \\ \partial_{\bc_M} J_u(\delta) \end{pmatrix}
		= \begin{pmatrix}
			-2\, \bbf(u) + \bA \balpha + \bA^\top \balpha + 2\, \bB  \bc\\
			-2\, \bg(u) + \bC \bc + \bC^\top \bc + 2\, \bB^\top \balpha
		\end{pmatrix} \\
		&= 2 \begin{pmatrix}
			-\bbf(u) + \bA \balpha +  \bB  \bc \\
			-\bg(u) + \bC \bc+  \bB^\top \balpha
		\end{pmatrix} 
		= - 2 \begin{pmatrix}
			\bbf(u)  \\
			\bg(u) 
		\end{pmatrix} 
		+ 	2 \begin{pmatrix}
			\bA			& \bB \\
			\bB^\top 	& \bC
		\end{pmatrix}
		\begin{pmatrix}
			\balpha \\
			\bc
		\end{pmatrix},
	\end{align*} 
	i.e., $\nabla \tilde{J}_u(\tilde\delta)=0$ if and only if \eqref{eq:LinSys}. Finally, 
	$\nabla^2 \tilde{J}_u(\tilde\delta) = 2 \begin{pmatrix}
			\bA 			& \bB \\
			\bB^\top 	& \bC
		\end{pmatrix}$ 
	is symmetric and since for $x \in \R^{M + N} $ 
	$$ x^{\top} \begin{pmatrix}
			\bA 			& \bB \\
			\bB^\top 	& \bC
		\end{pmatrix} x = \bigg\| \sum_{i=1}^N x_i \, \varphi_i
			+ \sum_{j=1+N}^M x_j \, v_j(a_j^\top\cdot+b_j) \bigg\|_0^2 \geq 0 ,$$
	$\nabla^2 \tilde{J}_u(\tilde\delta)$ is also positive semi-definite, which concludes the proof.
\end{proof}

This result shows that the coefficients are determined once directions and offsets are given. Setting  $\delta(u,\ba,\bb) := (\balpha^*(u,\ba,\bb), \ba, \bb, \bc^*(u,\ba,\bb))$, we consider the reduced cost function depending solely on directions and offsets
\begin{align}\label{eq:reducedcostfunction}
	\hat J_u:\R^{(d+1)M}\to\R_{\geq 0},
	\qquad
	\hat J_u(\ba,\bb) := \| u - u_{\delta(u,\ba,\bb)}\|_0^2
\end{align}
along with the reduced optimization problem
\begin{align}\label{eq:reducedoptproblem}
	\hat J_u (\hat\delta) \to \min!
	\qquad
	\hat\delta = (\ba,\bb) \in\R^{(d+1)M}.
\end{align}
Hence, we are left with finding the directions $\ba$ and the offsets $\bb$.

\section{Directions, offsets and particles}
\label{Sec:3}
Since the determination of directions and offsets amounts to solving a complex optimization problem, we aim at using a very well-known heuristic method, the particle swarm algorithm, see e.g.\ \cite{KennedyEberhart95, olsson2011particle}. In order to reduce computational complexity, we arrange our \enquote{particles} (which are associated to the collection of all directions $a_j\in\R^d$ and offsets $b_j\in\R$, $j=1,...,M$) in a dynamic grid. Hence, we call the arising scheme \emph{particle grid} method. 
 For each profile, we collect the direction and the offset in one vector $d_j:=(a_j,b_j)\in\R^{d+1}$. These vectors are then associated to some component $p_j\in (-1,1)^{D}=:\mS_{D}$. The vector $\bd=(d_j)_{j=1,...,M}\in\R^{DM}$ of all directions and offsets is then associated to one \emph{particle}. The dimension $D$ is at most $d+1$, but can also be smaller if the problem at hand fixes some components of $d_j$, see Example \ref{Ex:fixPart} below.

\begin{example}\label{Ex:fixPart}
	Consider the linear transport and wave equation in one space-dimension from Example \ref{Ex:1} and \ref{Ex:2}, respectively.  The underlying domain is $\Omega=\R^+\times\R$ and the variables read $(t,x)\in\Omega$ indicating time and space. Hence, in both cases, any direction takes the form $a=(a_t,a_x)$ and any ridge function reads $v(a_t\, t + a_x\, x+b)$ for some offset $b\in\R$. 
	
	The ridge functions should be consistent with the initial condition $u(0,x)=u_0(x)$, which implies that $v:=u_0$ and $a_x:= 1$ as well as $b:=0$. This shows that we do not need a full vector $d=(a_t,a_x,b)$ of dimension $d+1=3$, but that a single parameter suffices to represent each combination of direction and offset with a particle. 
	
In fact, in both cases, profiles take the form $v(x \pm \mu t)$, which means that we choose $a=(\pm\mu,1)^\top$ and $b=0$ with $\mu\in\R$.
\end{example}

The association of particles to all vectors $d_j$, $j=1,...,M$ (i.e., for all $M$ profiles), is done by constructing an appropriate mapping
 \begin{align*}
 	\pi: \mS_D \to \R^{d+1},
	\qquad
	\pi(p_j) =: (a_j,b_j)^\top.
 \end{align*}
Of course, on can construct several such mappings.

\begin{example}
For $D=d+1$, we frequently used the smooth transformation $\pi: (-1,1)^{d+1}\to\R^{d+1}$ defined as $\pi_i(s):= \tan(\tfrac{\pi}{2} s )$ for each component $i=1,...,d+1$. 
\end{example}
 
 \begin{example}[Example \ref{Ex:fixPart} continued]
 	In order to associate the real-valued parameter $\mu$ to a particle $p\in(-1,1)$, we can define $\pi: \mS_1\to\R^3$ by $p\mapsto (\tan(\tfrac{\pi}{2} {p}) ,1, 0)^\top$, i.e., $\mu=\tan(\tfrac{\pi}{2} {p})$.
\end{example}

This setting now allows us to reformulate the reduced optimization problem \eqref{eq:reducedoptproblem} in terms of the particles, i.e., $\hat J_u(\bp)\to\min!$ for $\bp=(p_1,...,p_M)\in\mS_D^{M}$, where 
\begin{align}\label{eq:CostFunProfiles}
	\hat J_u: \mS_D^{M}\to\R_{\geq 0},
	\quad
	\hat J_u(\bp):= \check J_u(\tau(p_1),...,\tau(p_M)),
\end{align}
where the arguments are to be understood in the following manner:
\begin{align*}
	\R^{(d+1)M} \ni \hat\delta
	=\begin{pmatrix}
		a_1 & a_2 & \cdots & a_M \\
		b_1 & b_2 & \cdots & b_M
	\end{pmatrix}
	=(d_1, ...,d_M) 
	= (\tau(p_1),...,\tau(p_M)).
\end{align*}
We call such a $\bp$ \emph{particle}.

\begin{example}\label{Ex:FindingDirections}
	In order to illustrate the challenges for solving the reduced optimization problem for \eqref{eq:CostFunProfiles}, let us consider three examples, where the function $u$ to be approximated is given as a sum of ridge functions (i.e., which can even be represented exactly). We consider the two-dimensional case $\Omega = (0,1)^2$ and $M=2$ profiles. The data is collected in the following table.
	
\begin{center}
	\begin{tabular}{r|l|l|l}
		Case & $u(x_1,x_2)$ & $v_1(\xi)$ & $v_2(\xi)$ \\[2pt] \hline 
		1 	& $1.6\, \cos(10x_1 + \frac{10}{3} x_2) + 0.8\, \cos (10x_1 -5x_2)$
			\vbox to 11pt {\vfil}
			& $\cos(10\xi)$
			& $\cos(10\xi)$ \\[3pt] 
		2	& $\cos(x_1+2x_2) + \cos(x_1-0.5\,x_2) $
			& $\cos(\xi)$
			& $\cos(\xi)$ \\[3pt]
		3 	& $5\, |x_1 - \frac12 x_2 - \frac12| + 0.6\, (x_1+\frac13 x_2)^2$
			& $|\xi-\frac12|$
			& $\xi^2$	
	\end{tabular}  
\end{center}
	
Given $v_1$ and $v_2$, setting $b_1=b_2=0$, $(a_1)_1=(a_2)_1=1$ we scatter $\mS_1^2$ by selecting particles $p=(p_1,p_2)\in (-1,1)^2$ for the remaining two unknowns $(a_1)_2$ and $(a_2)_2$. Each such particle defines a direction for which we define $u_\delta$ and compute the value of the cost function, i.e., the error $\| u-u_\delta\|_0$. The results are depicted in Figure \ref{Fig:1}. 

The functions $u$ are shown in the top row. The reduced cost function $\hat J_u$ as a function of $p\in(-1,1)^2$ is visualized in the second row, scattered on a grid of $129^2$ points. Solving for the two directions and offsets would thus amount finding a global minimum of the cost functions shown in the second row. As we see, we may face multiple local minima, even multiple global minima (in particular in cases 1 and 2, where $v_1=v_2$, so that we see symmetry), steep gradients and several highly localized phenomena. 
	\begin{figure}[!htb]
		\begin{center}
			\includegraphics[width=0.3\textwidth]{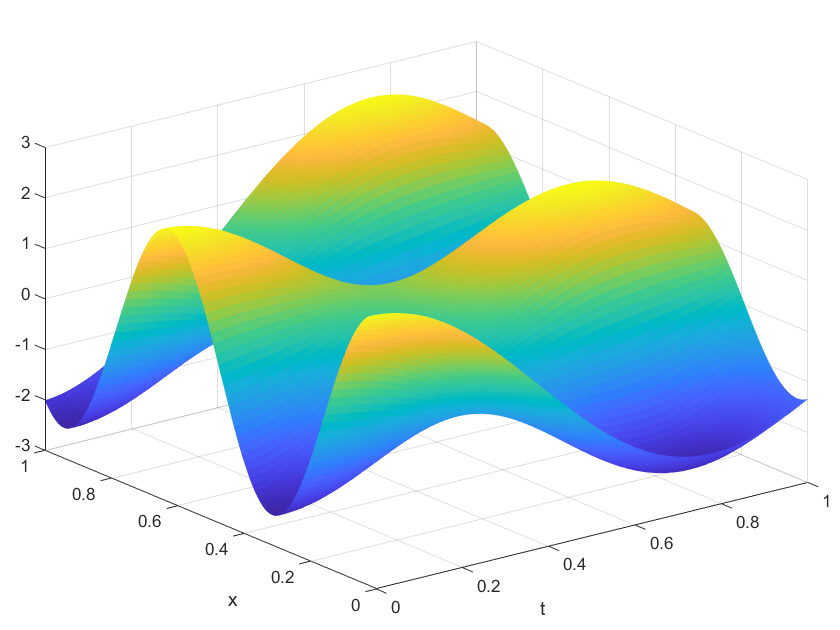}
			\includegraphics[width=0.3\textwidth]{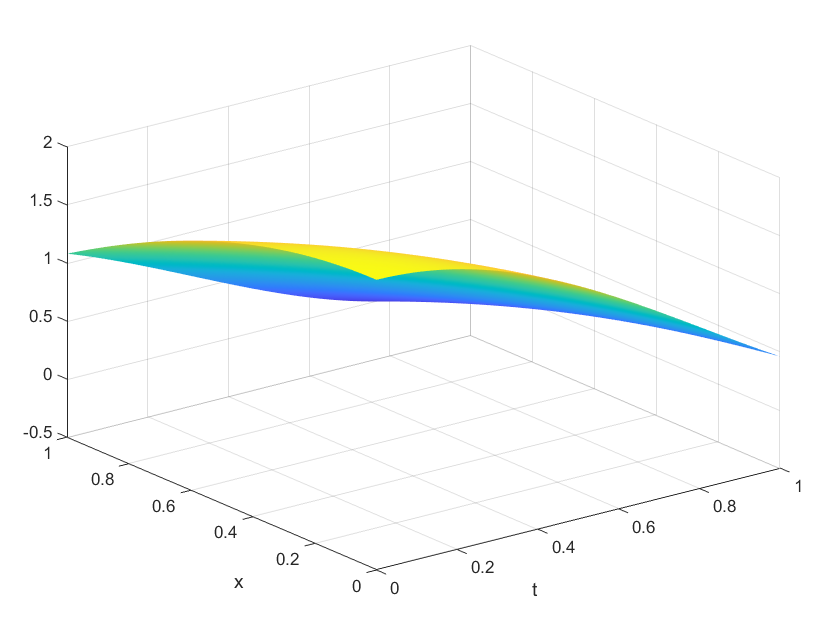}
			\includegraphics[width=0.3\textwidth]{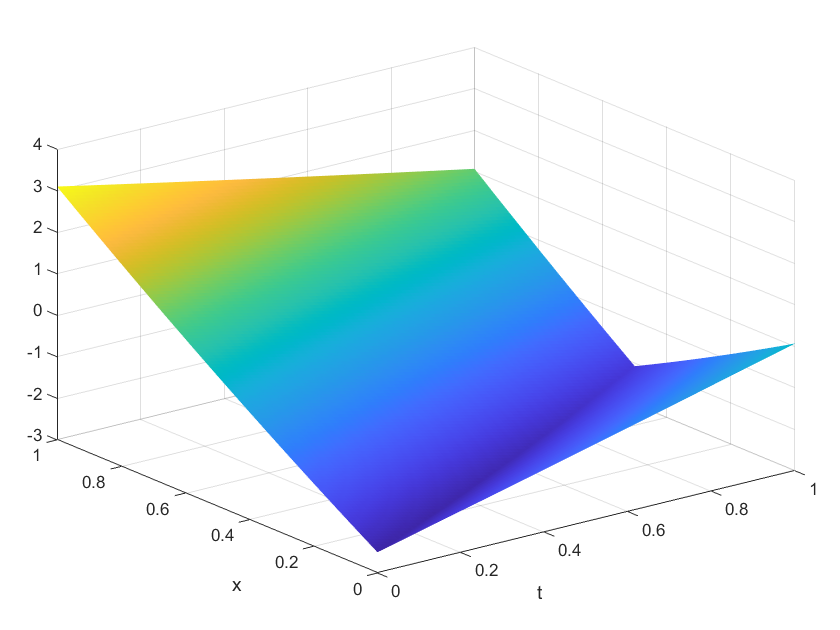}
	
			\includegraphics[width=0.3\textwidth]{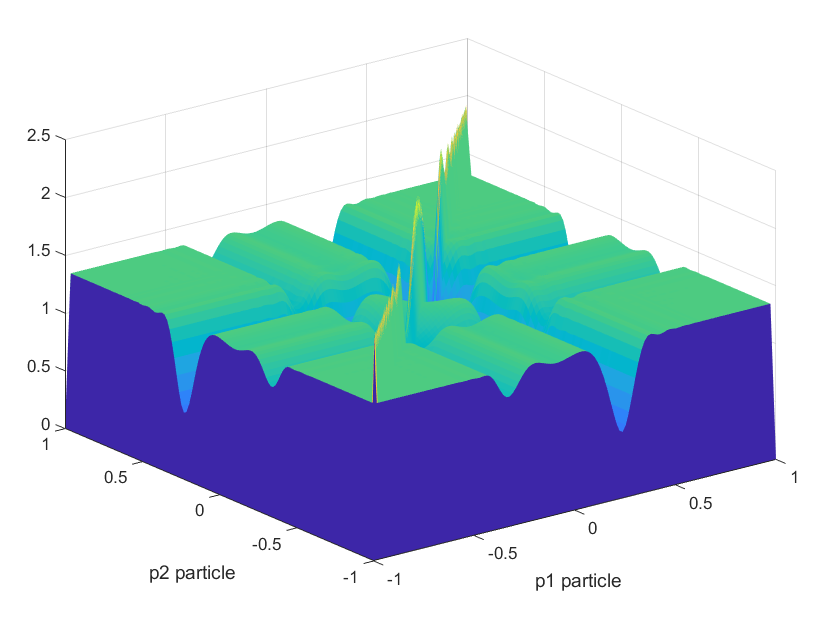}
			\includegraphics[width=0.3\textwidth]{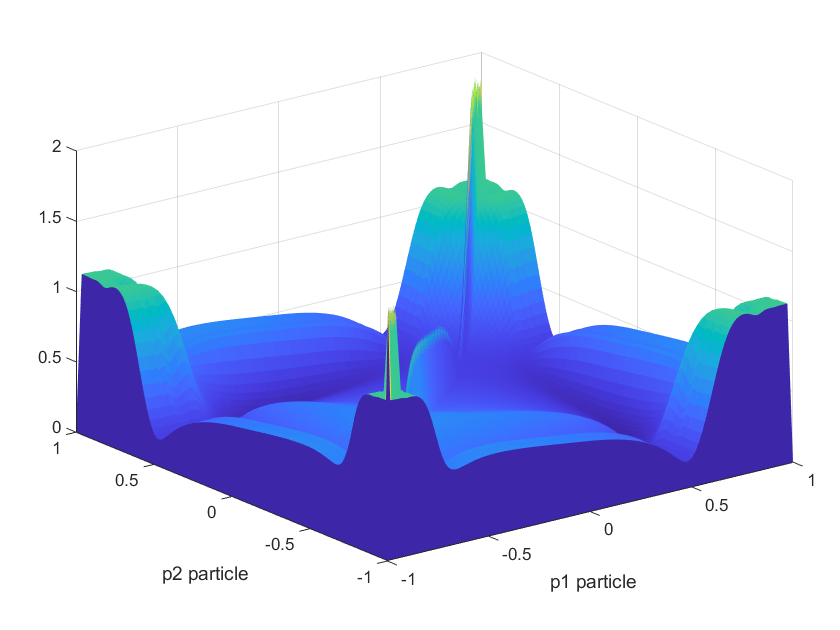}
			\includegraphics[width=0.3\textwidth]{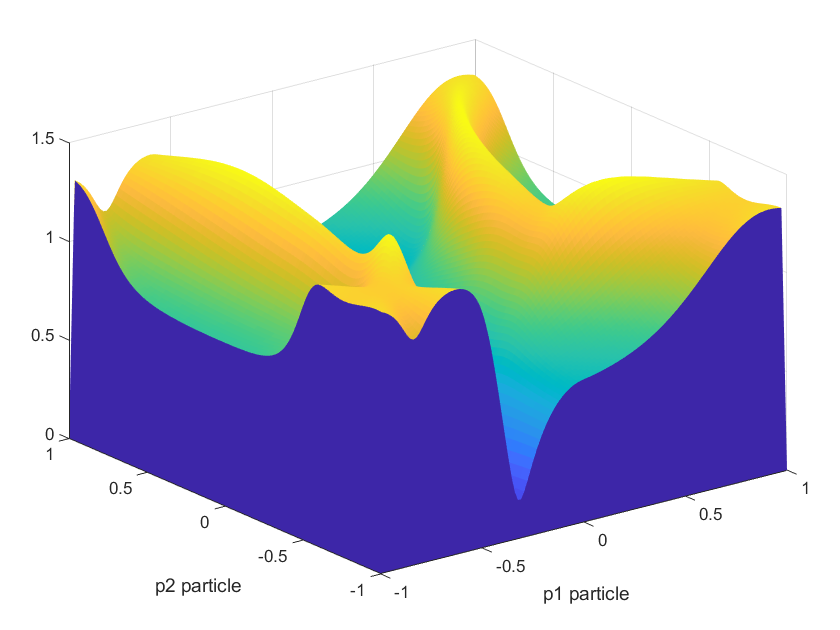}
			\caption{Cases 1-3 (from left to right): function to be approximated  (top row) and corresponding reduced cost functions (bottom row).\label{Fig:1}}
		\end{center}
	\end{figure}
\end{example}

As we see from Example \ref{Ex:FindingDirections}, we may face a truly complex optimization problem. In particular, we cannot hope to use gradient-based optimization techniques, at least not in a straightforward manner.

\section{A particle grid algorithm}
\label{Sec:4}
As already motivated earlier, we are now going to describe an algorithm which has shown good performance in determining at least a good approximation to a global minimum of $\hat J_u$ in \eqref{eq:CostFunProfiles}. Recall, that we are facing an optimization problem in $\mS_D^{M}\cong \mS_M^{D} \cong \mS_{DM} = (-1,1)^{DM}$. Hence, we define the optimization dimension as $P:=DM$.

The algorithm produces a sequence of \emph{particle grids}
\begin{align*}
	\bP^{(0)} \to \bP^{(1)} \to \cdots \to \bP^{(k)} \to \bP^{(k+1)} \to \cdots,
\end{align*}
where each grid (i.e., a swarm in form of a grid) $\bP^{(k)}$ consists of $\mpar$ particles in $\mS^P$. We choose $\npar$ nodes in each dimension, i.e.,  
\begin{align*}
	 \mpar = \npar^{P} 
	\quad\text{for}\quad \npar \in \N, 
\end{align*}
particles in $\mS^P=(-1,1)^{P}$. Then, we initialize the first particle swarm $\bP^{(0)}$ by taking the tensor product, yielding a regular grid. Each particle has uniquely defined next neighbors in each diagonal direction. This next neighbor relation does not change in the course of the iteration, as we will see in Lemma \ref{Lem:neighborrelation}. This means that each swarm is a \emph{grid} whose internal geometry does not change even if the position of each particle may vary. We may associate each particle grid $\bP^{(k)}$ with a tensor of dimension $P$ (e.g., a matrix for $P=2$). 

If $\bP^{(k)} = \{ \bp^{(k)}_\bi\in (-1,1)^{P}: \bi=(i_1,...,i_P)\in \{1,...,\npar \}^{P} \}$, and each particle takes the form $\bP^{(k)}\ni \bp^{(k)}_\bi =((\bp^{(k)}_\bi)_1,...,(\bp^{(k)}_\bi)_P)^\top\in (-1,1)^{P}$, the algorithm can be described as follows: 
Choose $\delta \in (0,\frac12)$. Then, for each particle $\bp^{(k)}_\bi\in \bP^{(k)}$, we consider the value of the particle and of the at most $3^{P}-1$ surrounding particles $\bp_\bj^{(k)}$ with $\bj=(j_1,...,j_P)$ where $|i_s - j_s | \leq 1$ for all $1 \leq s \leq P$. We collect the indices of these particles in a set $\cI(\bi) := \{ \bj=(j_1,...,j_P)\in \{ 1,...,n_p\}^P:\, |i_s - j_s | \leq 1 , \  1 \leq s \leq P \}$ and set
\begin{align}\label{smallesnb}
	\bq_\bi^{(k+1)} 
	:= \arg \min_{ \bj\in \cI(\bi)} \hat J_u( \bp_\bj^{(k)} ),
	\qquad
	|\cI(\bi)|\le 3^P-1.
\end{align}
Subsequently we get the next particle by the step:
\begin{equation}\label{stepparticle}
	\bp_\bi^{(k+1)} := (1-\delta) \bp_\bi^{(k)}  + \delta \bq_\bi^{(k+1)} .
\end{equation}

For the points on the boundary of the grid we need to slightly adapt this procedure. 
Technically, we view particles at opposite sides of the boundary as being neighbors, which means that, for example in one dimension the particle on the most left is considered as the neighbor of the particle on the most right. For dimension $1 \leq s \leq P$, a particle $\bp_\bj^{(k)}$ with $\bj=(j_1,...,j_{s-1},1,j_{s+1},...,j_P)$ is a neighbor of $\bp_\bi^{(k)}$ with $\bi=(j_1,...,j_{s-1}, \npar ,j_{s+1},...,j_P)$. This means that we are sticking the boundary together, so that we get a grid where in each dimension each particle has neighbors to the left and to the right, similar to a torus.

This concludes the description of one iteration $\bP^{(k)} \to \bP^{(k+1)}$, which is terminated after a predefined number $K\in\N$ of iterations with the output 
\begin{align}\label{outputapp}
	\bp_{\text{app}} := \arg \min_{ \bi\in \{1,...,n_p\}^P} \hat J_u(\bp_{\bi}^{(K)}).
\end{align}
which is used as an approximation for some minimizer $\bp^*$. 
Here, we often choose $\delta = \frac{1}{3}$, which turned out to be a reasonable choice as it always ensures a positive distance of neighboring particles. We summarize the method in Algorithm \ref{basisgen2}.
\begin{algorithm}[H] 
	\caption{Particle grid algorithm\label{basisgen2}}
	 \begin{algorithmic}[1]
	 	 \Require{Number of particles $\mpar = \npar^{P}$, maximal number of iterations $K$,\newline parameter $\delta\in (0,\frac12)$}
		 \State Initialize $\bP^{(0)} = \{ \bp^{(0)}_\bi: \bi\in \{1,..., \npar \}^P \}$ equidistantly
		 \For{$k = 0$ \textbf{to} $K-1$}
		 \For{all indices $\bi\in \{1,...,\npar\}^P$}
			\State Compute $\bq_\bi^{(k+1)}$ as in  \eqref{smallesnb}
		 	\State Calculate $\bp_\bi^{(k+1)}$ by \eqref{stepparticle};
				set $\bP^{(k+1)} = \{ \bp^{(k+1)}_\bi: \bi \in \{1,...,\npar \}^P\}$
			\State Compute $\hat J_u( \bp_\bi^{(k+1)} )$
		 \EndFor
		 \State Get $\bp_{\text{app}}$ as in formula \eqref{outputapp}
		 \EndFor
		 \Ensure{Approximation particle $\bp_{\text{app}}$}
	 \end{algorithmic}
\end{algorithm}

Figure \ref{Fig:ParticleSwarm} shows $6$ stages of Algorithm \ref{basisgen2} for one specific example with $P=DM=2$ and $\npar=10$. We start with a regular grid on top left and indicate how the algorithm moves the particles by showing the states for iterations $k=1$,  $k=5$, $k=25$,  $k=56$ and $k=90$. As we see, not all particles are concentrated in one point even for this case having a single global minimum. The reason is that sticking the opposite boundaries together prevents a concentration. However, most points are very close to the global minimum or on lines parallel to the coordinate axes through the point of global minimum.
 \begin{figure}[!htb]
		\begin{center}
				\includegraphics[width=0.30\textwidth]{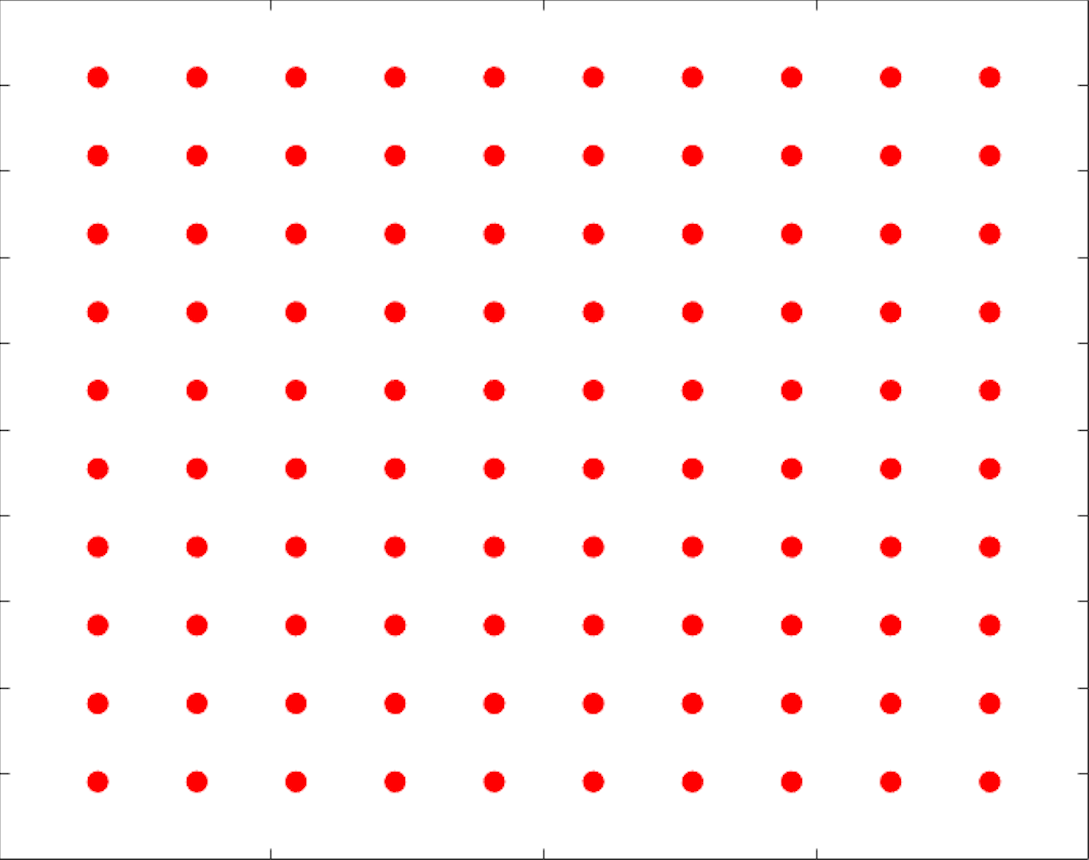}  \hspace{0.15cm}
				\includegraphics[width=0.30\textwidth]{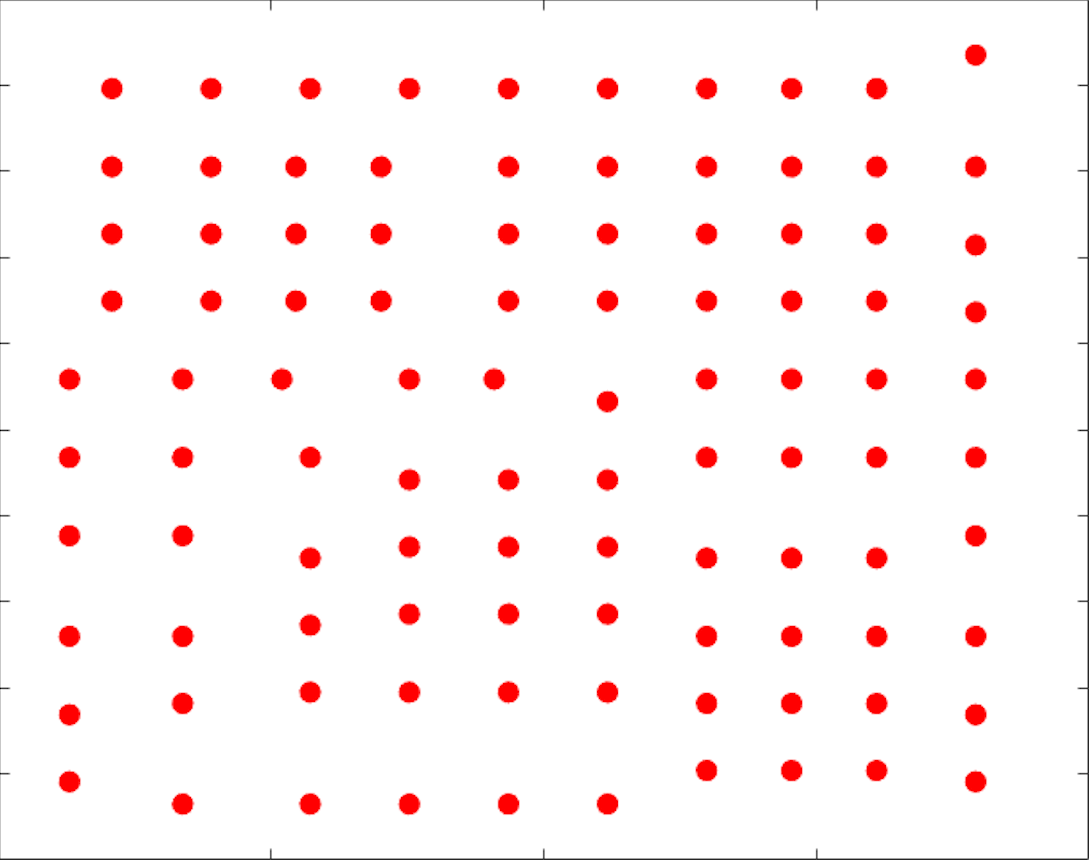}  \hspace{0.15cm}
				\includegraphics[width=0.30\textwidth]{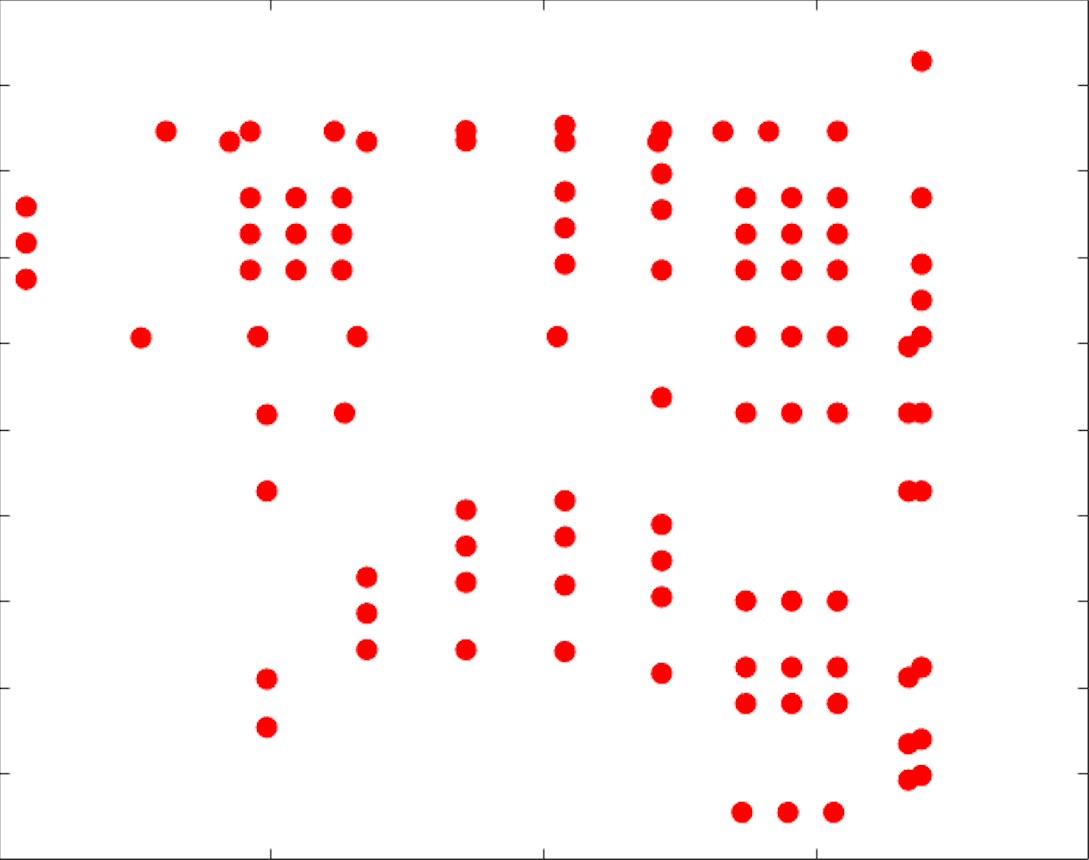}
				
				\includegraphics[width=0.30\textwidth]{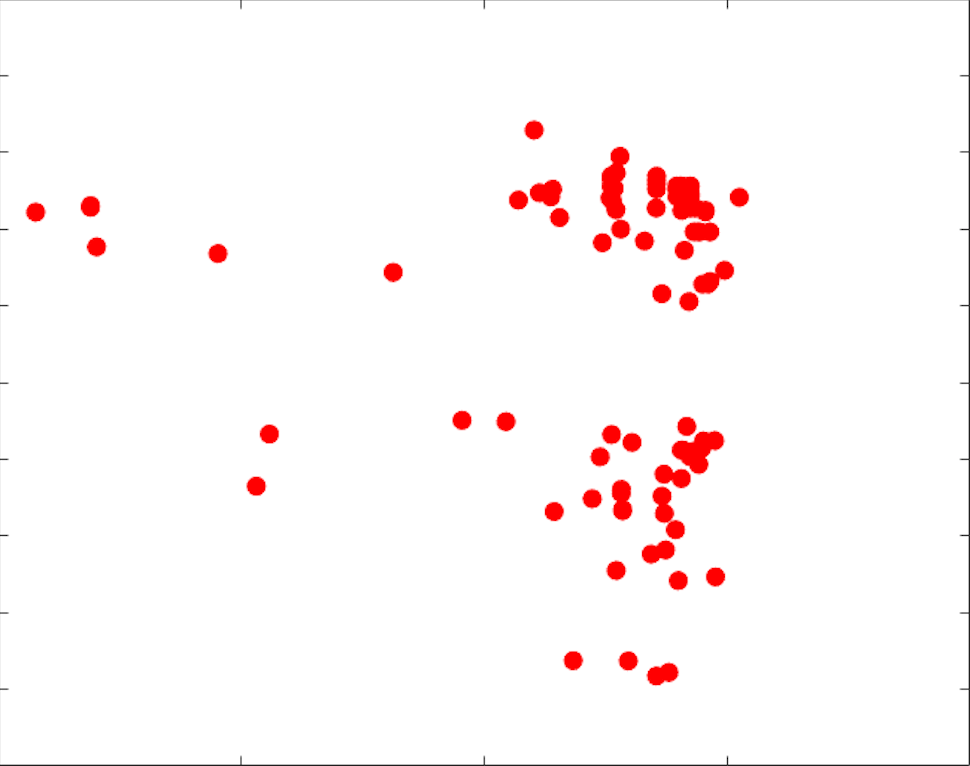} \hspace{0.15cm}
				\includegraphics[width=0.30\textwidth]{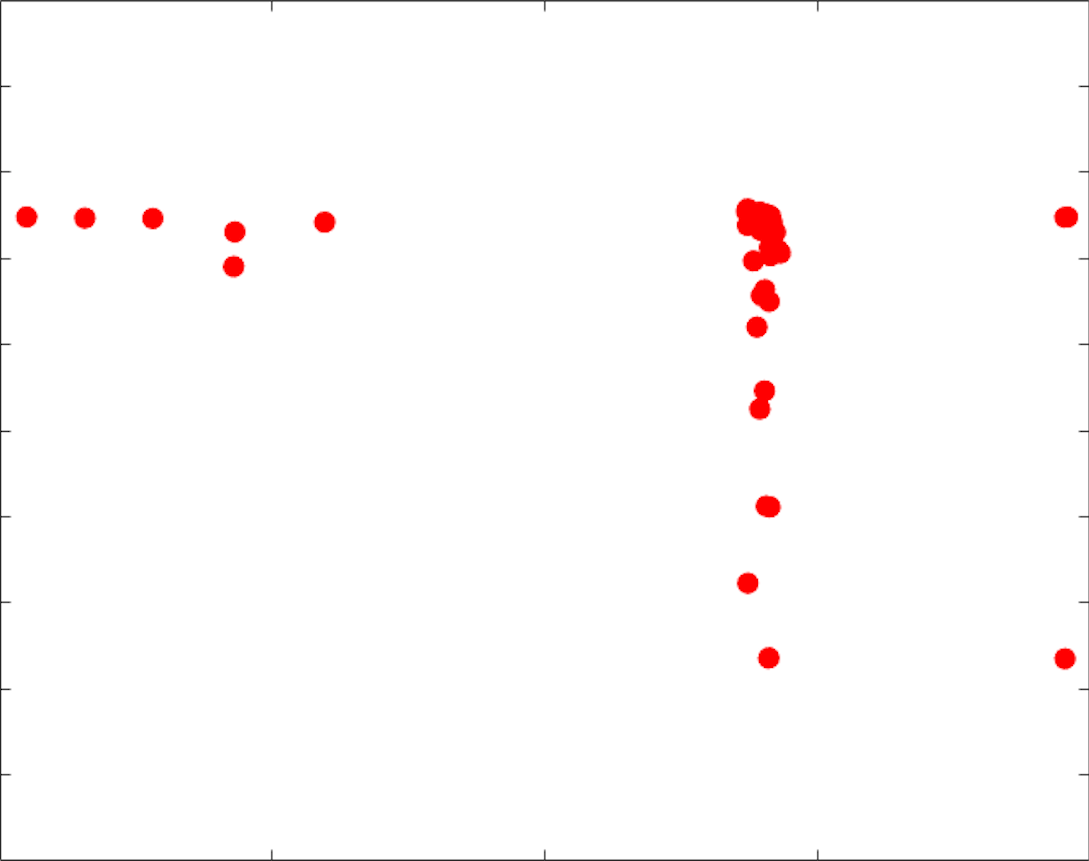} \hspace{0.15cm}
				\includegraphics[width=0.30\textwidth]{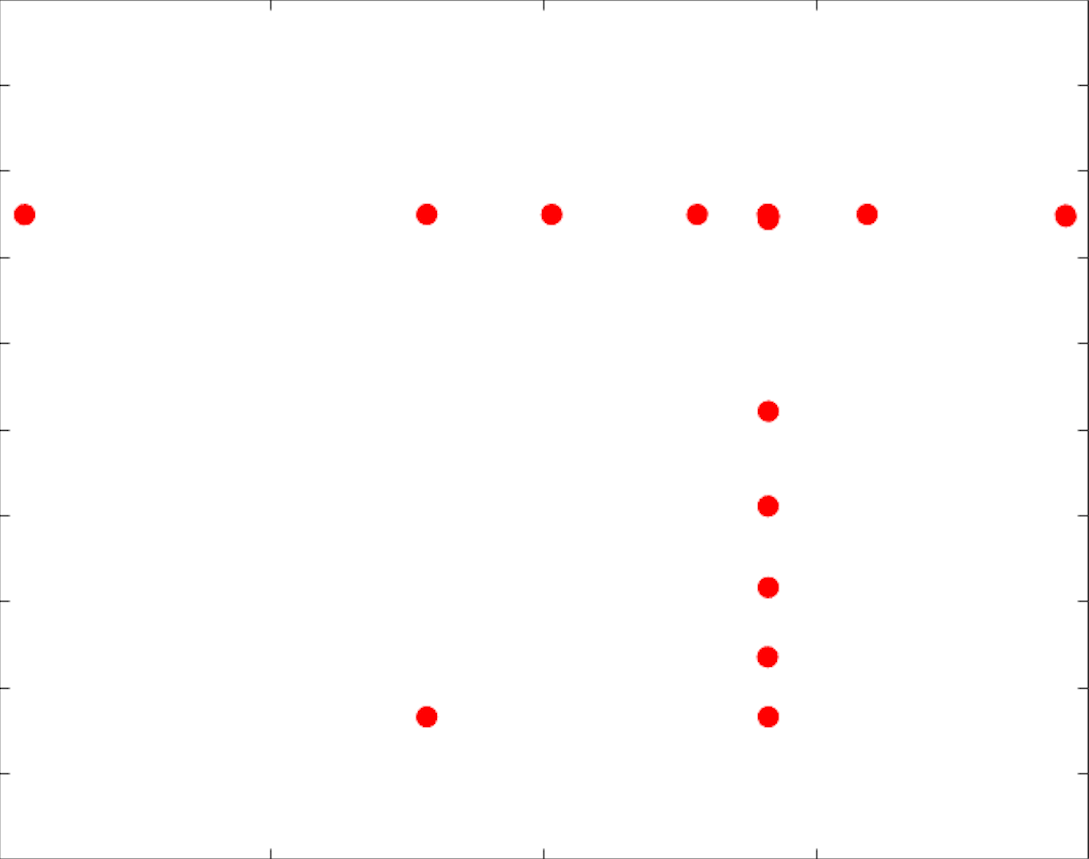}
			\caption{Particle grid iterations for $P=DM=2$, $\npar=10$: Iterations $k=0$, $k=1$,  $k=5$, $k=25$,  $k=56$,  $k=90$.
				\label{Fig:ParticleSwarm}}
		\end{center}
\end{figure}

The following observation is then immediate.
\begin{lemma} \label{Lem:neighborrelation}
The next neighbor relation does not change in the course of the iteration, i.e., the index set $\cI(\bi)$ is independent of the iteration $k$.\hfill\qed
\end{lemma}

\subsection{Complexity} 
Let $K$ be the number of iterations and $\mpar$ the number of particles. In order to compute $\hat J_u$, we sample $\Omega\subset\R^d$ by $n\in\N$ quadrature points for computing the $L_2$-norm. Hence, we need to compute point values of $u$ and $u_\delta$ at $n^d$ points. For each quadrature point, the evaluation of $u_\delta$ requires to evaluate $\varphi_i$, $i=1,...,N$ and $v_j(a_j^\top\cdot+b_j)$, $j=1,...,M$, which is a total of $\text{Cost}(\hat J_u) := N \cdot n^d + M \cdot (d+1) \cdot n^d$ operations for one evaluation of $J_u$ and $\mpar \cdot\text{Cost}(\hat J_u)$ evaluations for one particle grid $\bP^{(k)}$. For each iteration, we need to evaluate $\hat J_u$ for all grid points and per grid point, we have at most $3^P-1$ comparisons. 
In total, Algorithm \ref {basisgen2} thus requires in the order of 
\begin{align*}
	K\cdot \npar^P\cdot \big[ n^d\cdot (N + M\cdot (d+1)) + 3^P -1\big]
\end{align*}
operations, where $P=DM$ and $D\le d+1$. Hence, the algorithm gets to its limits for large particle dimension $D$ and large numbers $M$ of profiles. The dimension $N$ of the linear part only plays a minor role, especially as with sufficient memory the values of $\varphi_i$ can be computed once and be stored.

\subsection{Parallelization}
The particle grid algorithm has the advantage that it can easily be parallelized. In fact, in a shared memory environment, the current grid $\bP^{(k)}$  and any given particle $\bp_\bi^{(k)}$ is needed for determining the position $\bp_\bi^{(k+1)}$ in the next iteration. Hence, all such computations can be performed in parallel without further communication, which leads to linear speedup w.r.t.\ the numbers of processors. For distributed memory one would pass the positions of the neighbors of $\bp_\bi^{(k)}$ to the processor handling this particle.

\section{Numerical experiments}
\label{Sec:5}

In this section, we present results of some of our numerical examples. The main focus is the question how well a Linear/Ridge expansion is able to approximate certain functions and also the quantitative performance of the presented particle grid algorithm.

We start by applying our particle grid algorithm for determining a Linear/Ridge approximation of type \eqref{eq:approx}, i.e., we seek for an approximation in $U_{N,M}$ consisting of the sum of a linear combination of functions $\Phi_N$ and profiles $\cV_M$ for selected choices of such sets $\Phi_N$ and $\cV_M$. 

\subsection{Approximation properties}

First, we choose functions $u$ which can be written exactly in the form \eqref{eq:approx}, i.e., $u\in U_{N,M}$  and approximate the coefficients $\by$ in \eqref{eq:coeffs} by our particle grid algorithm. Doing so, we can monitor the error $\| u-u_\by\|_{0}$.

\subsubsection{Polynomials and Wavelets}
We start by a problem in two dimensions, $(t,x) \in \Omega = [0,1] \times [-1,1]$ which might be interpreted as time and space.  It seems to be a straightforward choice to use $\Span(\Phi_N)=\mathcal{P}_r(\Omega)$, i.e., the space of algebraic polynomials of degree at most $r\in\N$. Here we choose $r=2$, so that $N=6$ and for simplicity we use the monomial basis, i.e., 
\begin{equation*}
	\varphi_1(t,x) = 1, \ \varphi_2(t,x) = t, \ \varphi_3(t,x) = x, \ \varphi_4(t,x) = t^2, \ \varphi_5(t,x) = t x, \ \varphi_6(t,x) = x^2.
\end{equation*}
Of course, we could use a more stable basis $\Phi_6$ (e.g.\ orthonormal polynomials), but we are also interested to see how the algorithm can cope with ill-conditioned sets $\Phi_N$, which are possibly even allowed to be linearly dependent. 
 
 Concerning the profiles, we choose wavelets as it is known that dilates and translates of wavelets yield frames or bases of $L_2(\R^2)$. That makes them good candidates for profiles. Specifically, we take $M=4$ and
 \begin{compactitem}
 	\item $v_1$ as Mexican Hat, 
	\item $v_2$ as the Haar wavelet, 
	\item $v_3$ as Morlet wavelet,
	\item and $v_4$ to be the Hockeystick, which is also known as the ReLU activation function in neural networks,
 \end{compactitem}
see Figure \ref{Fig:3tr}. 
The function $u$ to be approximated is a Linear/Ridge expansion with $\alpha_i=1$, $i=1,...,6$, $b_j=0$, $c_j=1$, $j=1,...,4$ and $a_1 = (1,1)^{\top}$, $a_2 = (1,\sqrt{2}-1)^{\top}$, $a_3 = (1,-1)^{\top}$ and $a_4 = (1,\sqrt{2}+1)^{\top}$. The arising function is displayed in Figure \ref{Fig:3tl}. As we see, the shape of $u$ is rather complex. As the required discretization for performing computations, we use $h_t=\frac{1}{128}$, $h_x=\frac{1}{64}$, so that we have $129$ grid points for both variables. 
For Algorithm \ref{basisgen2} we choose as above $\delta=1/3$ and $\mpar = 6^4 = 1296$ particles. In Figure \ref{ErrorperIterations-a}, we monitor the $L_2$-error over the number of iterations. We obtain monotone convergence, without rate of course. For example, in order to reach an error smaller than $10^{-4}$ we need about 250 Iterations. After about 1000 iterations, the $L_2$-error is machine accuracy, i.e., $9.91\cdot 10^{-15}$. The approximation $u_\delta$ is shown in Figure \ref{Fig:3bl}.
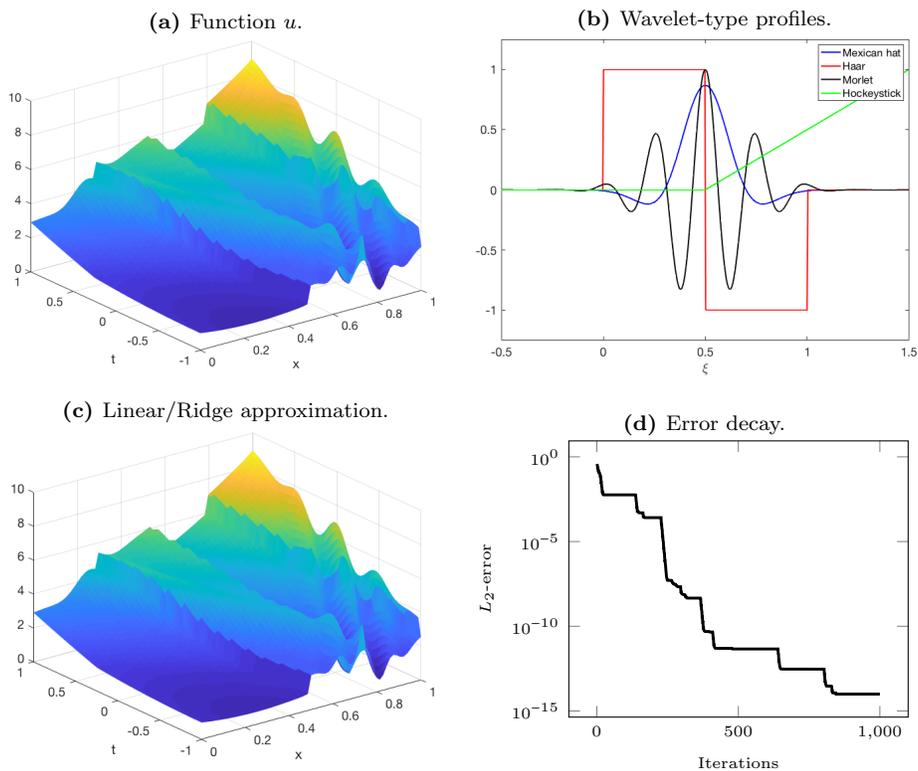
\begin{figure}[!htb]
	\begin{center}
		\begin{subfigure}[b]{0.49\textwidth}
			\caption{Function $u$.\label{Fig:3tl}}
			\includegraphics[width=0.99\textwidth]{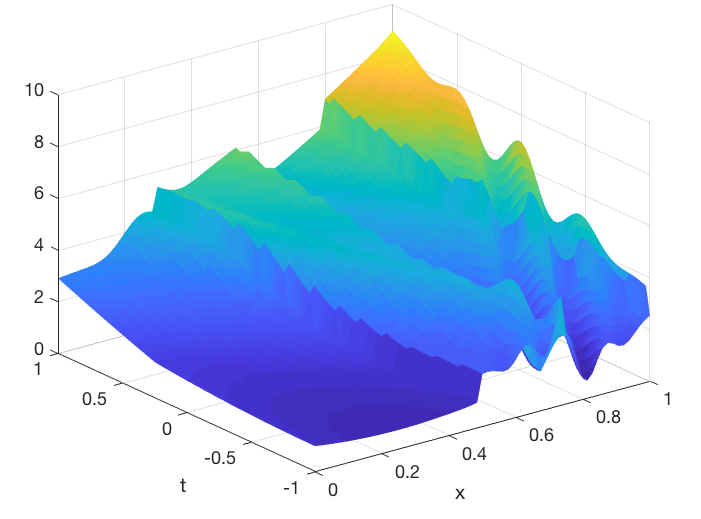}
		\end{subfigure}
		\hfil
		\begin{subfigure}[b]{0.49\textwidth}
			\caption{Wavelet-type profiles.\label{Fig:3tr}}
			\includegraphics[width=0.99\textwidth]{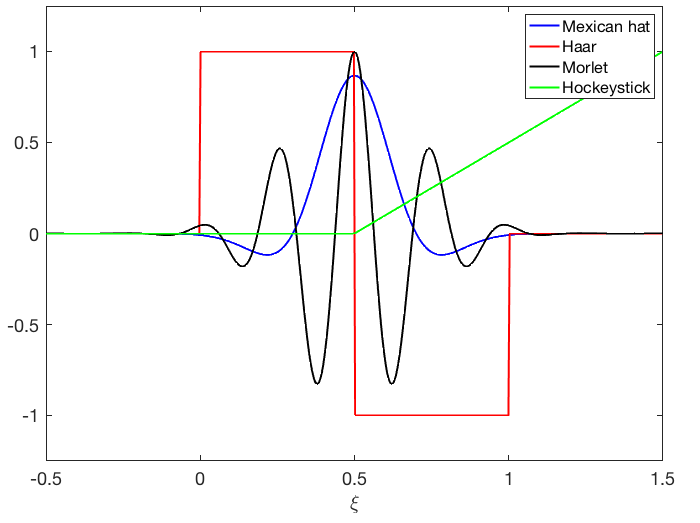}
		\end{subfigure}
		
		
		\begin{subfigure}[b]{0.49\textwidth}
			\caption{Linear/Ridge approximation.\label{Fig:3bl}}
			\includegraphics[width=0.99\textwidth]{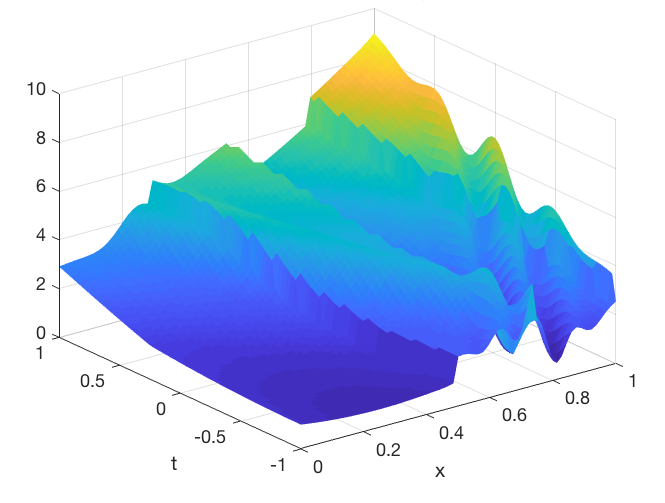}
		\end{subfigure}
		\hfil
		\begin{subfigure}[b]{0.49\textwidth}
		\caption{Error decay.\label{ErrorperIterations-a}}
		\pgfplotsset{every tick label/.append style={font=\tiny}}
			\begin{tikzpicture}[scale=0.99]
			\begin{axis}[
				width=0.99\textwidth,
				ymode=log,
				xlabel = {\tiny{Iterations}}, 
				ylabel={\tiny{$L_2$-error}},
				cycle list name=color list,
				unbounded coords=jump,
				legend pos=north east,
				legend cell align=left,
				legend style={font=\footnotesize},
				clip=false,
				]
				\addplot+[line width=1.1pt,black] table[y=L2error] {./Fig4a.txt};
				\end{axis}
			\end{tikzpicture}
		\end{subfigure}
	
		\caption{Approximation of function given as sum of quadratic polynomials ($N=6$) and $M=4$ wavelets of different type.}
		\label{NumExampleswavelet}
	\end{center}
\end{figure}

\subsubsection{Discontinuous staircase function}
Even though the shape of the function in our first example is complex, it is a smooth function\footnote{Of course, the Haar wavelet is discontinuous, but a quite easy and single function to be approximated.}. Hence, we are now going to consider a function which consists of eight jumps along straight lines which are rotated in order to exclude a tensor product approximation. We use again $\Omega = (0,1) \times (-1,1)$. The profiles are chosen as the jump function displayed in Figure \ref{NumExamplesjumps} top right. In order to investigate the role of non-linearly independence of the profiles, we take $v_1=\cdots =v_8 = \mathbbm{1}_{\R^+}$, i.e., $M=8$ identical profiles. We forego the linear part here, i.e., $N=0$. 
The function $u$ to be approximated takes the form \eqref{eq:approx} with $N=0$ and $M=8$ choosing $b_j=0$, $c_j =1$, $j=1,...,8$ and directions according to the angles $\frac{\pi}{9}$, so that we obtain a staircase-like function as displayed in Figure \ref{NumExamplesjumps} bottom left.

For the discretization, we choose as above $129$ grid points in both coordinate directions. For the particle grid algorithm, we use $\delta=1/3$, $\mpar=3^8=6561$ and $\mpar=4^8=65563$ particles, respectively, where we note that the algorithm did not converge for $2^8$ particles. The error decay is shown in Figure  \ref{ErrorperIterations-b}. The reason why we show results for $\mpar=3^8$ and $\mpar=4^8$ is the fact that we observe a stagnation of the error for $\mpar=3^8$ after $50$ iterations, whereas $\mpar=4^8$ yields machine accuracy after 49 iterations. This shows that one might be forced to use a large number of particles in order to reach high accuracies. 
After 50 iterations, we achieved an $L_2$-error of $0$. In Figure \ref{NumExamplesjumps}, we show the staircase-type step function that is also the computed approximation after 50 iterations.
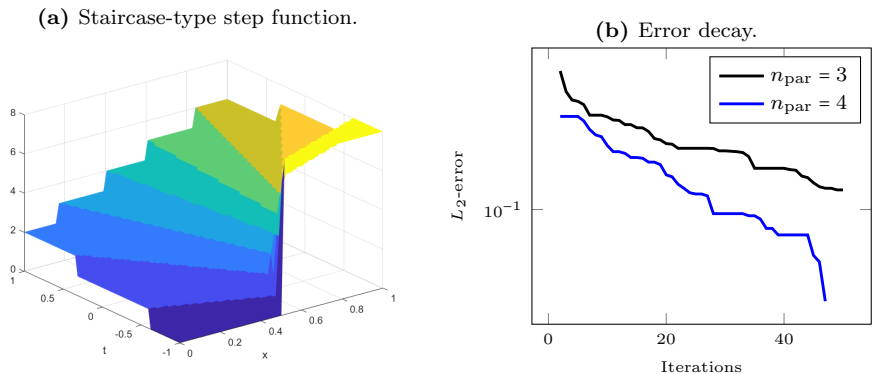
\begin{figure}[!htb]
	\begin{center}
		\begin{subfigure}[b]{0.49\textwidth}
		\caption{Staircase-type step function.\label{NumExamplesjumps}}
		\includegraphics[width=0.99\textwidth]{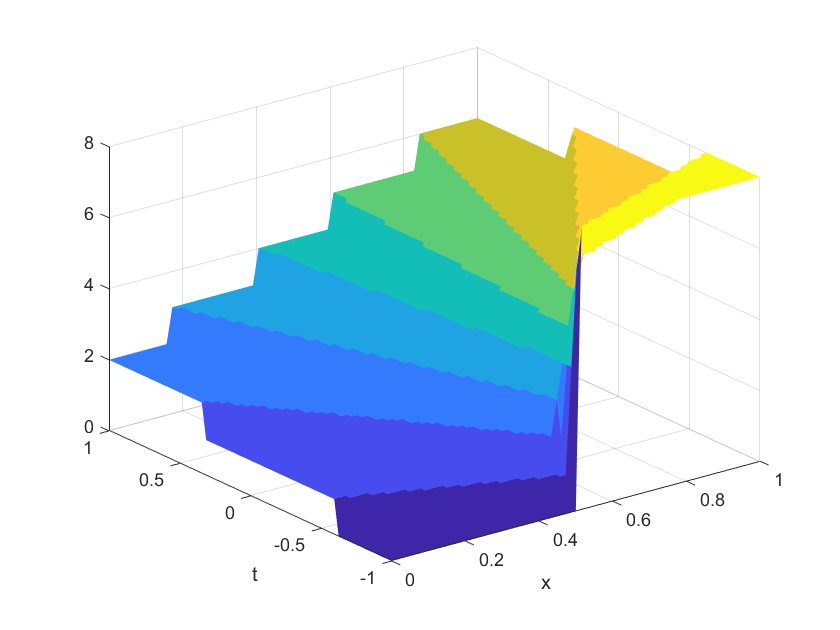}
		\end{subfigure}
		\hfil
		\begin{subfigure}[b]{0.49\textwidth}
		\caption{Error decay.\label{ErrorperIterations-b}}
			\pgfplotsset{every tick label/.append style={font=\tiny}}
			\begin{tikzpicture}[scale=0.99]
			\begin{axis}[
				width=0.99\textwidth,
				ymode=log,
				xlabel = {\tiny{Iterations}}, 
				ylabel={\tiny{$L_2$-error}},
				cycle list name=color list,
				unbounded coords=jump,
				legend pos=north east,
				legend cell align=left,
				legend style={font=\footnotesize},
				clip=false,
				]
				\addplot+[line width=1.1pt,black] table[y=npar4,select coords between index={1}{49}] {./Fig4b.txt};
				\addplot+[line width=1.1pt,blue] table[y=npar5,select coords between index={1}{46}] {./Fig4b.txt};
				\addlegendentry{$\npar=3$}
				\addlegendentry{$\npar=4$}
			\end{axis}
			\end{tikzpicture}
		\end{subfigure}

		\caption{Approximation of a given staircase-type step function.}
		\label{ErrorperIterations}
	\end{center}
\end{figure}

\subsection{Non-exact approximation functions}

So far, we used the algorithm to approximate functions that can exactly be represented in terms of the chosen families $\Phi_N$ and $\cV_M$. This shows how fast (or slow) the algorithm is able to find the function in terms of directions and offsets. Now, we are going to consider functions $u$ that cannot be represented exactly by $u_{\delta}$, i.e., $\inf_{\delta} \| u - u_{\delta}\|>0$. To this end, consider for $c \geq 1$ the function 
\begin{equation}
	u(t,x; \mu) := \sum_{k=1}^{\infty} \frac{1}{k!} \cos\big[  2 \pi ( k \tfrac{\sqrt{c}}{10}  \,  t +  x)  \big],
	\qquad (t,x) \in \Omega := (0,1) \times (-1,1).
\end{equation}
For the linear part of the approximation, we choose \emph{snapshots} $\varphi_j(t,x) = u(t,x;j)$, for $j=1,...,N$, and the profiles are chosen as $v_i :=\cos(2\pi \cdot)$, $i=1,...,M$ for $M\in\{0,...,4\}$. In Figure \ref{ErrorperIterations-2}, we display the $L_2$-error for $u(t,x) = u(t,x; 100)$ and different values of $N$ and $M$. For the discretization, we choose as above 129 grid points in both coordinate
directions and for the particle grid algorithm, we used  $\delta = 1/3$, $\mpar = 5^M$ particles. We do not obtain monotone convergence as $N$ grows, which might be a stability issue. On the other hand, however, the error decreases for fixed $N$ and increasing $M$ in a monotonic manner.
\begin{figure}[!htb]
	\begin{center}
			\begin{tikzpicture}[scale=0.95]
			\begin{axis}[
				width=0.99\textwidth,
				height=0.35\textheight,
				ymode=log,
				xlabel = $N$, 
				ylabel=$L_2$-error,
				cycle list name=color list,
				unbounded coords=jump,
				legend pos=north east,
				legend cell align=left,
				legend style={font=\footnotesize},
				clip=false,
				]
				\addplot+[line width=1.5pt,black,mark=x] table[y=0] {./Fig6.txt};
				\addplot+[line width=1.5pt,blue,mark=x] table[y=1] {./Fig6.txt};
				\addplot+[line width=1.5pt,red,mark=x] table[y=2] {./Fig6.txt};
				\addplot+[line width=1.5pt,violet,mark=x] table[y=3] {./Fig6.txt};
				\addplot+[line width=1.5pt,green,mark=x] table[y=4] {./Fig6.txt};
				\addlegendentry{$M=0$}
				\addlegendentry{$M=1$}
				\addlegendentry{$M=2$}
				\addlegendentry{$M=3$}
				\addlegendentry{$M=4$}
			\end{axis}
			\end{tikzpicture}
		\caption{Approximation of an infinite series: $L_2$-error for different sizes of $\Phi_N$ and $\cV_M$.}
		\label{ErrorperIterations-2}
	\end{center}
\end{figure}
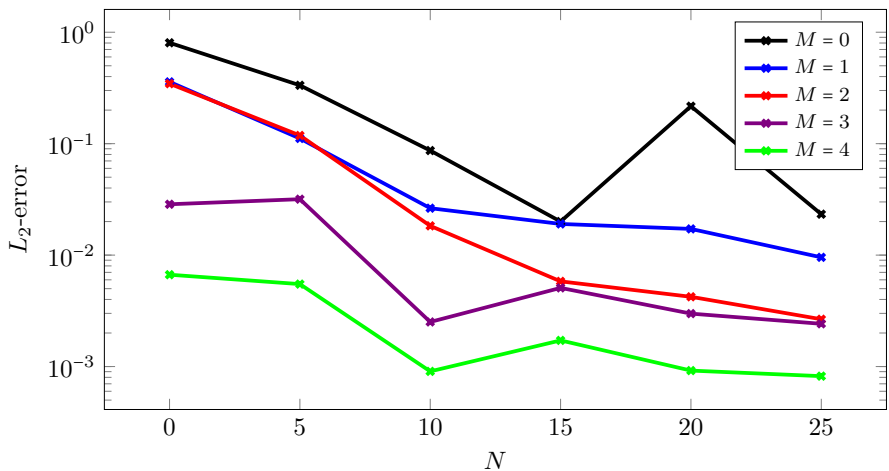

\subsection{Parametric Partial Differential Equations (PPDEs)}\label{Sec:PPDE}

Next, we consider two PPDEs that depend on parameters $\mu\in\R^P$ and investigate the approximation of the solution $u(\mu)$ for various parameter values. The first problem is a stationary one, the second is instationary; both are defined in $\Omega= (0,1)^2$, which is interpreted as a time/space-domain for the wave equation. For both cases, we use $\delta=1/3$ and $\mpar =  11^2 = 121$ particles.

\subsubsection{Case 1: Thermal block } \label{thermal}
The thermal block is a classical problem for model reduction, \cite{Haasdonk:RB}. We consider the stationary case, which is a Poisson problem with piecewise constant coefficients, which serve as parameters. In fact, we decompose $\overline{\Omega}$ by $\overline{\Omega}_i := [0,1] \times [\frac{i-1}{4}, \frac{i}4]$, $i=1,...,4$. The bilinear form of the variational formulation of the problem reads for $\mu=(\mu_1,...,\mu_4)\in\R^4$, $P=4$, as follows
\begin{align*}
	a(u,v;\mu) = \sum_{i=1}^{4} \mu_i \int_{\Omega_i} \nabla u(x)\, \nabla v(x)\, dx.
\end{align*} 
The right-hand side and boundary data can be retrieved from \cite{Haasdonk:RB}, and the formula for the exact solution can easily be seen to be for $x=(x_1,x_2)\in\Omega$ as follows
\begin{align*}
	u(x; \mu) =  \sum_{i=1}^4 \frac{1}{\mu_i} \varphi_i(x) ,  
	\quad  
	\varphi_i(x) 
	=   \begin{cases}
		 \frac{1}{4} , 	& 0 \le x_2 < \frac{i-1}{4}  ,  \\
		-x_2 + \frac{i}{4}, 	& x_2 \in [ \frac{i-1}{4}, \frac{i}{4}] ,  \\
		0, 			& \frac{i}{4} < x_2 \leq 1,
	\end{cases}    
\end{align*}
see the left part of Figure \ref{Fig:ExSolPDE}. 
We consider different parameters $\mu \in [0.1,10]^4$.

Due to the specific form of the solution, we use the linear approximation with $N=4$ and choose $\Phi_4 := \{ \varphi_1,  \varphi_2,  \varphi_3,  \varphi_4 \}$ with the functions defined above. We took $M=2$ arbitrarily chosen profiles. The reason for this choice is as follows: Since the solution can be approximated quite accurately by the linear combination of the \emph{snapshots} $\varphi_i$ (which is due to the fact that such snapshots are known from the Reduced Basis Method to yield a very accurate approximation), the particle grid algorithm should automatically detect that no profiles are needed. As we can see from the results in Table \ref{table_thermalwave} this is in fact the case -- the algorithm finds the solution up to machine accuracy after only $1$ iteration of the particle grid algorithm. We stress the fact that we tested much more parameter values and always observed this behavior.

\begin{figure}[!htb]
	\begin{center}
			\includegraphics[height=0.25\textheight]{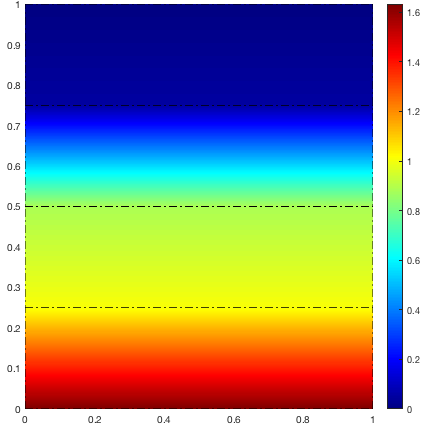} \hspace{0.7cm}
			\includegraphics[height=0.25\textheight]{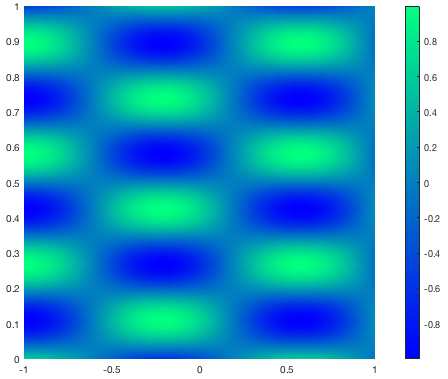}
			\caption{Exact solutions for the PDE examples 1 and 3. Left: $4 \times 1$ thermal block for $\mu=(0.4, 2, 0.3, 5)$. Right: Solution to wave equation for $\mu=0.8$.}
			\label{Fig:ExSolPDE}
	\end{center}
\end{figure}

\subsubsection{Case 2: Linear transport equation} \label{transport}
Recalling Example \ref{Ex:1}, we consider the parametric linear transport equation with initial condition $u_0$, which we choose as a single profile, i.e., $M=1$. For the linear part we use the same $\Phi_4$ as in \S\ref{thermal}, i.e., the algorithm is supposed to set the corresponding coefficients $\alpha_i$ to zero.

\subsubsection{Case 3: The wave equation} \label{wave}
As in Example \ref{Ex:2}, we consider the univariate parametric wave equation $u_{tt}(t,x) - \mu^2 \, u_{xx}(t,x)=0$ for $(t,x)\in\Omega$ along with initial conditions. We choose the initial data in such a manner that the solution reads
\begin{align*}
u(t,x; \mu) 
   &= 0.5  \,  \sin(10 x + 10 \mu t) + 0.5  \,  \sin(10 x - 10 \mu  t),
\end{align*}
which is a superposition of two ridge functions. Consequently, we set $M=2$ and $\cV_2  :=\{ v_1, v_2 \}  = \{  \sin(10 \cdot), \sin(10 \cdot ) \}$. Again, for the linear part we use $\Phi_4$ as in \S\ref{thermal}.

\begin{table}[!htb] 
	\begin{center}
	\begin{tabular}{r|l|l|r|l}
		Case & PPDE & parameter $\mu$ &  no.\ iterat.\ $K$ & $L_2$-error \\[2pt] \hline\hline
		1 	& Thermal block & $(0.1,10,1,0.6)$ & $1$ & $5.1019e-15$ \\[2pt]
		1 	& Thermal block & $(10,2,0.1,0.5)$ & $1$ & $1.4446e-14$ \\[2pt]
		1 	& Thermal block & $(0.4,2,0.3,5)$ & $1$ & $6.0861e-15$ \\[2pt] \hline
		2 	& Transport & $1/4$ & $12$ & $7.1348e-05$  \\[2pt]
		2 	& Transport & $1/4$ & $ 66$ & $4.6205e-16$  \\[2pt]  \hdashline
		2 	& Transport & $1$ & $19$ & $3.0119e-05$  \\[2pt]
		2 	& Transport & $1$ & $70$ & $9.3829e-17$  \\[2pt] \hdashline
		2 	& Transport & $4$ & $24$ & $6.1985e-05$  \\[2pt]
		2 	& Transport & $4$ & $67$ & $0$  \\[2pt]  \hline
		3 	& Wave & $1/4$ & $ 20$ & $7.6682e-05$  \\[2pt]
		3 	& Wave & $1/4$ & $ 83$ & $8.9850e-16$  \\[2pt]  \hdashline
		3 	& Wave & $1$ & $21$ & $8.2400e-05$  \\[2pt]
		3 	& Wave & $1$ & $86$ & $3.1765e-16$  \\[2pt] \hdashline
		3 	& Wave & $4$ & $28$ & $4.3012e-05$  \\[2pt]
		3 	& Wave & $4$ & $83$ & $8.9850e-16$ \\ \hline
	\end{tabular}
	\end{center}
	\caption{Parametric PDEs: Errors and iterations for both examples and different parameter values.  \label{table_thermalwave}}
\end{table}

The results are shown in Table \ref{table_thermalwave}. As expected, the hyperbolic wave equation is a harder problem than linear transport, which in turn is much harder problem than the thermal block. The numbers of the required iterations of the particle grid algorithm are significantly higher in order to reach a desired tolerance.  On the other hand, we see that the algorithm is able to dismiss all the linear functions. The algorithm in fact converges and even reaches machine accuracy -- at the expense of more iterations.

\subsection{Higher dimensions}
Next, we are considering a problem in higher dimensions, namely travelling plane waves (i.e., a wave function that is constant over any plane that is orthogonal to a fixed direction in space). The general form of a plane wave thus reads for a given normalized direction $n = (n_1,n_2,n_3)^{\top} \in \R^3$, $\|n\| = 1$ and some velocity $c$ as follows
\begin{align*}
	A(t, x ) := v \big( n^{\top} x -c \ t  \big), \quad v:\R \to \R, 
\end{align*}
which is actually a ridge function. 
It is well-known that $A$ can also be obtained as the solution of the 3d-wave equation $A_{tt} - c^2 \Delta A=0$, which is the connection to our previous example. 
From a physical point of view, we shall assume that the plane waves emerge from two different light sources, i.e., for fixed $t$ it takes the form 
\begin{align*} 
	u(x) := \sum_{i=1}^2 v_i \big( n_i^{\top} x -c \ t  \big),
\end{align*}
which we use as function to be approximated. We choose the specific profiles $v_1 := v_2 := \sin$, i.e., again, identical profiles. We fix time and velocity as $t=1$, $c=1$. Doing so, we get the directions $\ba\in \R^6$ for $d=3$, $M=2$, i.e., $D=nM=6$ as $a_j=n^j\in\R^3$, $j=1,2$ and $\bb=(b_1,b_2)$, $b_1=b_2=-c$. We use the specific choices  $n_1 := ( -1/\sqrt{2}, 1/\sqrt{2}, 0)$ and $n_2 := (0,- 1/\sqrt{2}, 1/\sqrt{2})$.

We are interested in the convergence history of the particle grid algorithm. To this end, we fix the number of particles $\npar=4$ per direction, i.e., a total of $\mpar=\npar^D = 4^6=4096$. The results are shown in Figure \ref{3dtravellingcap}. We obtain quite fast convergence at the early stages which then slows down.
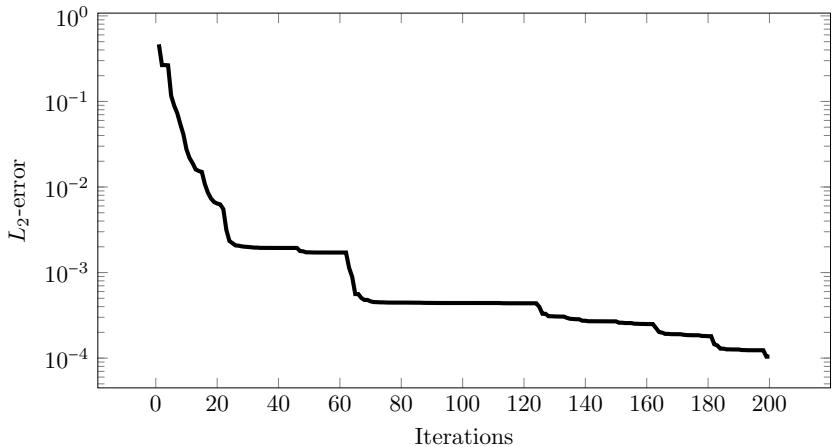
\begin{figure}[!htb] 
	\begin{center}
			\begin{tikzpicture}[scale=0.89]
			\begin{axis}[
				width=0.99\textwidth,
				height=0.35\textheight,
				ymode=log,
				xlabel = Iterations, 
				ylabel=$L_2$-error,
				cycle list name=color list,
				unbounded coords=jump,
				legend pos=north east,
				legend cell align=left,
				legend style={font=\footnotesize},
				clip=false,
				]
				\addplot+[line width=1.75pt,black] table[y=L2error] {./Fig8.txt};
				\end{axis}
			\end{tikzpicture}
		\caption{3d planar wave. $L_2$-error over iterations for fixed number $\npar=69$ of particles. The minimal error is $1.0404 10^{-4}$.} \label{3dtravellingcap}
	\end{center}
\end{figure}

\subsection{Training directions for parametric PDEs}
Let us reconsider parameter-dependent PDEs as in \S\ref{Sec:PPDE} above. Within a model reduction context, a reduced model is typically determined offline within a training phase. In our setting the sets $\Phi_N$ and $\mathcal{V}_M$ would thus be determined in an offline training.

In an online environment (typically suitable for multi-query or realtime situations), new parameters $\mu\in\R^P$ are given and an approximation is to be determined extremely fast -- in online complexity. If we would determine the online approximation in terms of $u_\delta$ in \eqref{eq:approx}, we would need to determine (optimal) directions that are parameter-dependent, i.e., $\ba(\mu)$ and $\bb(\mu)$, which might be computationally costly in the sense that many iterations of the particle grid algorithm would be needed, destroying online efficiency.

In order to speedup this procedure, we suggest to train the mapping $\mu \mapsto (\ba(\mu), \bb(\mu))$ offline as follows: For a given (or to be determined) set $\{ \mu_1,..., \mu_{n_{\text{train}}} \}$ of  training parameters, we determine (highly accurate) approximations $\ba(\mu_i), \bb(\mu_i)$, $i=1,...,n_{\text{train}}$, offline by the particle grid algorithm. Then, we determine a interpolation $\mu \mapsto (\check \ba(\mu), \check \bb(\mu))$ of those offline data such that the interpolation error $\|(\ba(\mu), \bb(\mu))-(\check{\ba}(\mu), \check{\bb}(\mu))\|$ is small, at least for a number of test samples $\mu$. Of course, this interpolation error influences the overall online approximation error. We are going to investigate this influence.

To this end, we consider the solution $u(\cdot,\cdot\,;\mu)$ of the parametric linear transport problem $u_t + \mu^{-1/2} \, u_x = 0$, $u(0)=u_0$, since it is known that this problem also results in poor approximation rates using standard linear model reduction such as the Reduced Basis Method, \cite{OR16}. In addition, the function $\mu\mapsto \mu^{-1/2}$ is much harder for the interpolation as a polynomial parameter-dependence. We choose $10$ equidistant training  parameters $\mu_i\in [0.1,1]$. Since the offset is zero here, we only determined the corresponding directions $a(\mu_i)$, $i=1,...,10$, which are real numbers here. The left graph in Figure \ref{fig:interp_mu} shows a cubic spline interpolation of the obtained data. We also display the exact curve $\mu\mapsto a(\mu)$ by determining high resolution approximations of optimal directions for $100$ parameters $\mu$ by the particle grid algorithm. As we see, the error is almost negligible.

Next, we computed an online approximation for a new parameter $\mu$ as follows: (1) Evaluate the interpolation to retrieve $\check{a}_j(\mu)$, set $b_j=0$, $N=0$; (2) compute the coefficients $c_j$ and obtain an approximation $\check u_\delta$ in \eqref{eq:approx}. This is to be compared with the function $u_\delta$ using the exact direction $a_j(\mu)$. On the right in Figure \ref{fig:interp_mu}, we display these errors $\| \check u_\delta(\cdot,\cdot;\mu)- u_\delta(\cdot,\cdot;\mu)\|_{0}$ for different initial conditions $u_0$ which also serve as the single profile. We used the same knots for the interpolation and see that the quantitative errors depend on the initial condition or more precisely on the derivative of the profile. However, please note the range of the vertical axis, which indicates that all errors are in fact in a comparable range.

\begin{figure}[!htb]
	\begin{center}
      \includegraphics[width=0.452\textwidth]{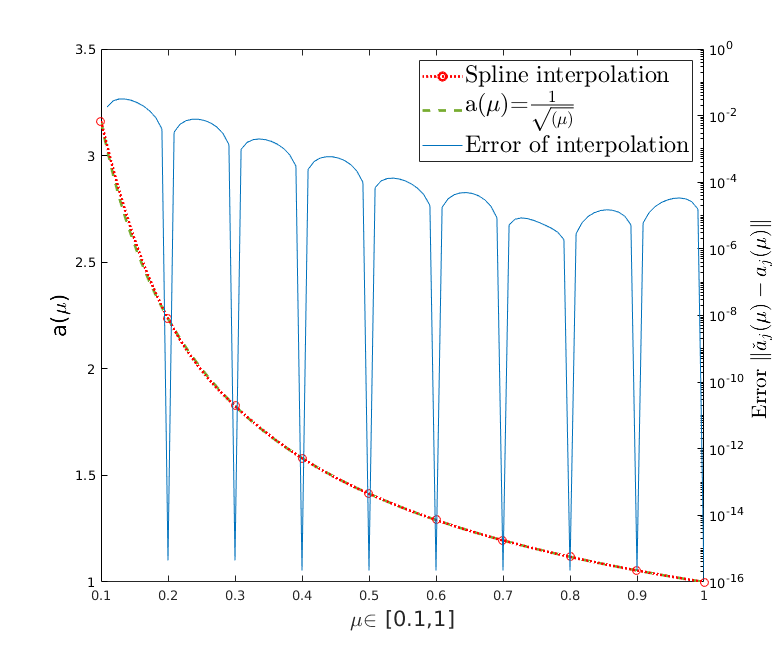}
      \includegraphics[width=0.49\textwidth]{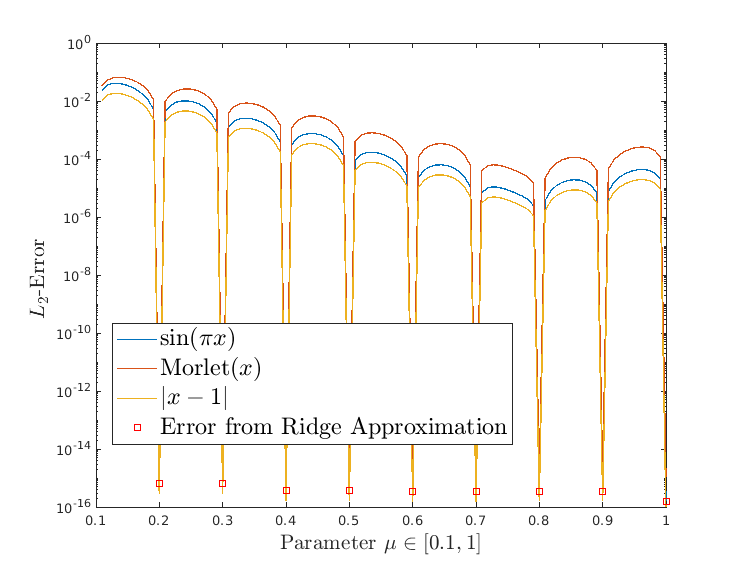}
		\caption{Training directions for the parametric linear transport equation: Spline interpolation of directions (left) and errors for the obtained ridge approximation (right).}
		\label{fig:interp_mu}
	\end{center}
\end{figure}

\subsection{Conclusions}
As we did much more experiments than we can report here (due to page limitation) let us collect some observations that we have seen and our corresponding conclusions.
\begin{compactitem}
	\item Even though we have seen that the optimization problem arising from the Linear/Ridge approximation is a challenging task, the particle grid algorithm often works very well.
	\item For approximating a given function, the performance seems to be better for smooth functions. However, the algorithm yields also good results for  non-continuous or multivariate functions, even though machine accuracy is harder to reach then. 
	\item By treating time \enquote{just as another variable}, the presented approach can handle stationary and instationary problems in the same manner.
	\item Choosing the number of particles sufficiently large, we were always able to reach machine accuracy.
	\item The algorithm is able to detect if a linear approximation is already sufficient to reach a desired accuracy. Hence, the scheme is robust in the considered cases of fast and of slow decay of the Kolmogorov $N$-width. We anticipate that this is restricted to (P)PDEs with linear characteristics.
\end{compactitem}

\section{Outlook}
\label{Sec:6}
The above described results of our numerical experiments seem to indicate that this path might be continued. Of course, we are aware that research in several directions is required, e.g.
\begin{compactitem}
	\item the introduced particle grid algorithm is based upon the particle swarm heuristics. There is no rigoros convergence analysis, which is a significant drawback in particular compared to linear RBMs, where online efficiency and a posteriori error control is certified. 	
	One might think of using (stochastic) gradient-descent methods in combination with backpropagation/automatic differentiation as known from the training of neural networks. Another option might be to start by the particle swarm algorithm and then use the result as starting point for a decent method. 
	\item in the current form, the particle grid algorithm is not yet online efficient, which would be required for using it within a multi-query and/or realtime environment. Also here, techniques from neural networks might help.
	\item last, but not least, we did not focus on the training of $\Phi_N$ and $\cV_M$, but merely viewed them as being given. 
\end{compactitem}

\bibliographystyle{abbrv}
\bibliography{./GJU.bib}
\end{document}